\documentclass{amsart}

\usepackage{amsfonts,amssymb,verbatim,amsmath,amsthm,latexsym,textcomp,amscd}
\usepackage{latexsym,amsfonts,amssymb,epsfig,verbatim}
\usepackage{amsmath,amsthm,amssymb,latexsym,graphics,textcomp, hyperref}
\usepackage{graphicx}
\usepackage{color}
\usepackage{url}

\begin{document}

\newtheorem{theorem}{Theorem}[section]
\newtheorem{prop}[theorem]{Proposition}
\newtheorem{lemma}[theorem]{Lemma}
\newtheorem{cor}[theorem]{Corollary}
\newtheorem{definition}[theorem]{Definition}
\newtheorem{defn}[theorem]{Definition}
\newtheorem{conj}[theorem]{Conjecture}
\newtheorem{claim}[theorem]{Claim}
\newtheorem{defth}[theorem]{Definition-Theorem}
\newtheorem{obs}[theorem]{Observation}
\newtheorem{rmark}[theorem]{Remark}
\newtheorem{qn}[theorem]{Question}
\newtheorem{theo}[theorem]{Theorem}
\newtheorem{thmbis}{Theorem}
\newtheorem{dfn}[theorem]{Definition} 
\newtheorem{defi}[theorem]{Definition} 
\newtheorem{coro}[theorem]{Corollary}
\newtheorem{corbis}{Corollary}
\newtheorem{propbis}[thmbis]{Proposition} 
\newtheorem*{prop*}{Proposition} 
\newtheorem{lem}[theorem]{Lemma} 
\newtheorem{lembis}[thmbis]{Lemma} 
\newtheorem{claimbis}{Claim} 
\newtheorem{fact}[theorem]{Fact} 
\newtheorem{factbis}{Fact} 
\newtheorem{qst}[theorem]{Question} 
\newtheorem{qstbis}{Question} 
\newtheorem{pb}[theorem]{Problem} 
\newtheorem{pbbis}{Problem} 
 \newtheorem{question}[theorem]{Question}
\newtheorem{rem}[theorem]{Remark}
\newtheorem{remark}[theorem]{Remark}
\newtheorem{rmk}[theorem]{Remark}
\newtheorem{example}[theorem]{Example}
\newtheorem{eg}[theorem]{Example}
\newtheorem{notation}[theorem]{Notation}
\newenvironment{preuve}[1][Preuve]{\begin{proof}[#1]}{\end{proof}}

\newcommand{\hhat}{\widehat}
\newcommand{\boundary}{\partial}
\newcommand{\C}{{\mathbb C}}
\newcommand{\integers}{{\mathbb Z}}
\newcommand{\natls}{{\mathbb N}}
\newcommand{\bbN}{{\mathbb N}}
\newcommand{\ratls}{{\mathbb Q}}
\newcommand{\reals}{{\mathbb R}}
\newcommand{\bbR}{{\mathbb R}}
\newcommand{\lhp}{{\mathbb L}}
\newcommand{\tube}{{\mathbb T}}
\newcommand{\cusp}{{\mathbb P}}
\newcommand\AAA{{\mathcal A}}
\newcommand\BB{{\mathcal B}}
\newcommand\CC{{\mathcal C}}
\newcommand\DD{{\mathcal D}}
\newcommand\EE{{\mathcal E}}
\newcommand\FF{{\mathcal F}}
\newcommand\GG{{\mathcal G}}
\newcommand\HH{{\mathcal H}}
\newcommand\II{{\mathcal I}}
\newcommand\JJ{{\mathcal J}}
\newcommand\KK{{\mathcal K}}
\newcommand\LL{{\mathcal L}}
\newcommand\MM{{\mathcal M}}
\newcommand\NN{{\mathcal N}}
\newcommand\OO{{\mathcal O}}
\newcommand\PP{{\mathcal P}}
\newcommand\QQ{{\mathcal Q}}
\newcommand\RR{{\mathcal R}}
\newcommand\SSS{{\mathcal S}}
\newcommand\TT{{\mathcal T}}
\newcommand\UU{{\mathcal U}}
\newcommand\VV{{\mathcal V}}
\newcommand\WW{{\mathcal W}}
\newcommand\XX{{\mathcal X}}
\newcommand\YY{{\mathcal Y}}
\newcommand\ZZ{{\mathcal Z}}
\newcommand\CH{{\CC\Hyp}}
\newcommand\MF{{\MM\FF}}
\newcommand\PMF{{\PP\kern-2pt\MM\FF}}
\newcommand\ML{{\MM\LL}}
\newcommand\PML{{\PP\kern-2pt\MM\LL}}
\newcommand\GL{{\GG\LL}}
\newcommand\Pol{{\mathcal P}}
\newcommand\half{{\textstyle{\frac12}}}
\newcommand\Half{{\frac12}}
\newcommand\Mod{\operatorname{Mod}}
\newcommand\Area{\operatorname{Area}}
\newcommand\ep{\epsilon}
\newcommand\Hypat{\widehat}
\newcommand\Proj{{\mathbf P}}
\newcommand\U{{\mathbf U}}
 \newcommand\Hyp{{\mathbf H}}
\newcommand\D{{\mathbf D}}
\newcommand\Z{{\mathbb Z}}
\newcommand\R{{\mathbb R}}
\newcommand\Q{{\mathbb Q}}
\newcommand\E{{\mathbb E}}
\newcommand\til{\widetilde}
\newcommand\length{\operatorname{length}}
\newcommand\tr{\operatorname{tr}}
\newcommand\gesim{\succ}
\newcommand\lesim{\prec}
\newcommand\simle{\lesim}
\newcommand\simge{\gesim}
\newcommand{\simmult}{\asymp}
\newcommand{\simadd}{\mathrel{\overset{\text{\tiny $+$}}{\sim}}}
\newcommand{\ssm}{\setminus}
\newcommand{\pair}[1]{\langle #1\rangle}
\newcommand{\T}{{\mathbf T}}
\newcommand{\inj}{\operatorname{inj}}
\newcommand{\collar}{\operatorname{\mathbf{collar}}}
\newcommand{\bcollar}{\operatorname{\overline{\mathbf{collar}}}}
\newcommand{\I}{{\mathbf I}}

\newcommand{\bbar}{\overline}
\newcommand{\UML}{\operatorname{\UU\MM\LL}}
\newcommand{\EL}{\mathcal{EL}}
\newcommand\MT{{\mathbb T}}
\newcommand\Teich{{\mathcal T}}

\makeatletter
\@tfor\next:=abcdefghijklmnopqrstuvwxyzABCDEFGHIJKLMNOPQRSTUVWXYZ\do{%
  \def\command@factory#1{%
    \expandafter\def\csname cal#1\endcsname{\mathcal{#1}}
    \expandafter\def\csname frak#1\endcsname{\mathfrak{#1}}
    \expandafter\def\csname scr#1\endcsname{\mathscr{#1}}
    \expandafter\def\csname bb#1\endcsname{\mathbb{#1}}
    \expandafter\def\csname rm#1\endcsname{\mathrm{#1}}
  }
 \expandafter\command@factory\next
}
\makeatother

\newcommand*{\longhookrightarrow}{\ensuremath{\lhook\joinrel\relbar\joinrel\rightarrow}}
\newcommand {\tto}{ \to\rangle }
\newcommand {\onto} {\twoheadrightarrow}
\newcommand {\into} {\hookrightarrow}
\newcommand {\xra} {\xrightarrow}    
\newcommand{\ra}{\rightarrow}
\newcommand{\imp} {\Rightarrow}
\newcommand{\actedon}{\curvearrowleft} 
\newcommand{\actson}{\curvearrowright}

\newcommand {\sd} {\rtimes}   
\newcommand{\semidirect}{\ltimes}
\newcommand{\isemidirect}{\rtimes}
\newcommand{\tensor}{\otimes}
\newcommand{\wreath}{\Lbag}

\newcommand{\du}{\sqcup}
\newcommand{\Dunion}{\bigsqcup} 
\newcommand{\disjoint}{\sqcup}
\newcommand{\normal} {\vartriangleleft}

\newcommand {\ie}{ i.e.  }

\newcommand{\ul}[1]{\underline{#1}} 
\newcommand{\ol}[1]{\overline{#1}}


\newcommand{\Cay}{\operatorname{Cay}}
\newcommand{\proj}{\operatorname{proj}}
\newcommand{\Fix} {\operatorname{Fix}}
\newcommand{\dist}{\operatorname{dist}}
\newcommand{\diam}{\mathop{\mathrm{diam}\;}}


\title{Height, Graded Relative Hyperbolicity and Quasiconvexity}

\author[Francois Dahmani]{Francois Dahmani}

\address{Institut Fourier,
Universite  Grenoble Alpes,
100 rue des Maths, CS 40700,
F-38058 Grenoble Cedex 9,
France}

\email{francois.dahmani@univ-grenoble-alpes.fr }

\author[Mahan Mj]{Mahan Mj}

\address{Tata Institute of Fundamental Research. 1, Homi Bhabha Road, Mumbai-400005, India}

\email{mahan@math.tifr.res.in}
\email{mahan.mj@gmail.com}


\thanks{The first author acknowledge support of ANR grant
  ANR-11-BS01-013, and from the Institut Universitaire de France. The
  research of the second author is partially supported by  a DST J C Bose Fellowship. }   

\date{\today}

 \begin{abstract}
 We introduce the notions of geometric height and graded (geometric) relative hyperbolicity in this paper. We use these to characterize
 quasiconvexity in hyperbolic groups, relative  quasiconvexity in relatively hyperbolic groups, and convex cocompactness in mapping class groups and $Out(F_n)$.
\end{abstract}

\maketitle

\tableofcontents

\section{Introduction}
  It is well-known that quasiconvex subgroups of hyperbolic groups have
  finite height.  In order to distinguish this notion from the notion of
  {\bf geometric height} introduced later in this paper, we shall call
  the former {\bf algebraic height}: Let $G$ be a finitely generated
  group and $H$ a subgroup.  We say that a collection of conjugates
  $\{ g_iHg_i^{-1} \}, i =1, \dots, n$ are {\bf essentially distinct} if
  the cosets $\{ g_i H\}$ are distinct.  
  We
  say that $H$ has finite {\bf algebraic height} if there exists
  $n \in \mathbb{N}$ such that the intersection of any $(n+1)$
  essentially distinct conjugates of $H$ is finite. The minimal $n$ for
  which this happens is called the {\bf algebraic height} of $H$. 
  Thus $H$ has algebraic height one if and only if it is almost malnormal.
  This
  admits a natural (and obvious) generalization to a finite collection
  of subgroups $H_i$ instead of one $H$.  Thus, if $G$ is a hyperbolic
  group and $H$ a quasiconvex subgroup (or more generally if
  $H_1, \dots, H_n$ are quasiconvex), then $H$ (or more generally the
  collection $\{ H_1, \dots, H_n\}$) has finite algebraic height
  \cite{GMRS}. (See \cite{D03, hruska-wise} for generalizations to the
  context of relatively hyperbolic groups.)  Swarup asked if the
  converse is true:
  
  \begin{qn}\cite{bestvinahp}
     Let $G$ be a hyperbolic group and $H$ a finitely generated
     subgroup. If $H$ has finite height, is $H$
     quasiconvex? \label{swarup} 
  \end{qn}
 
   An example of an infinitely generated (and hence non-quasiconvex)
   malnormal subgroup of a finitely generated free group was obtained in
   \cite{dm} showing that the hypothesis that $H$ is finitely generated
   cannot be relaxed. On the other hand, Bowditch shows in
   \cite{bowditch-relhyp} (see also \cite[Proposition
   2.10]{mahan-relrig}) the following positive result:
 
  \begin{theorem} \cite{bowditch-relhyp} 
     Let $G$ be a hyperbolic group
     and $H$ a subgroup.  Then $G$ is strongly relatively hyperbolic with
     respect to $H$ if and only if $H$ is an almost malnormal quasiconvex
     subgroup. \label{rhqcchar} 
  \end{theorem}

  One of the motivational points for this paper is to extend Theorem
  \ref{rhqcchar} to give a characterization of quasiconvex subgroups of
  hyperbolic groups in terms of a notion of {\it graded relative
    hyperbolicity} defined as follows:

  \begin{defn} \label{grh} 
    Let $G$ be a finitely generated group, $d$
    the word metric with respect to a finite generating set and $H$ a
    subgroup.  Let $\HH_i$ be the collection of intersections of $i$
    essentially distinct conjugates of $H$, let $(\HH_i)_0$ be a choice of
    conjugacy representatives, and let $\CH_i$ be the set of cosets of
    elements of $(\HH_i)_0$.   
    Let $d_i$ be the metric on $(G,d)$ obtained by
    electrifying\footnote{The second author acknowledges the moderating
      influence of the first author on the more extremist terminology
      {\it electrocution} \cite{mahan-split, mahan-ibdd}} the elements
    of $\CH_i$. We say that $G$ is {\bf graded relatively hyperbolic}
    with respect to $H$ (or equivalently that the pair $(G,\{H\})$ has
    graded relative hyperbolicity) if

    \begin{enumerate}
        \item $H$ has algebraic height $n$ for some $n \in \natls$.
        \item Each element $K$ of $\HH_{i-1}$ has a finite relative
          generating set $S_K$, relative to
          $H\cap \HH_i (:= \{ H\cap H_i : H_i \in \HH_i\})$. Further, the
          cardinality of the generating set $S_K$ is bounded by a number depending only on $i$ (and
          not on $K$).
        \item $(G,d_i)$ is strongly hyperbolic relative to $\HH_{i-1}$,
          where each element $K$ of $\HH_{i-1}$ is equipped with the word
          metric coming from $S_K$.
      \end{enumerate}
    \end{defn}

    The following is the main Theorem of this paper (see Theorem
    \ref{hypcharzn} for a more precise statement using the notion of {\it
      graded geometric relative hyperbolicity} defined later) 
    providing a
    partial positive answer to Question \ref{swarup} and a generalization
    of Theorem \ref{rhqcchar}:

    \begin{theorem} Let $(G,d)$ be one of the following:
      \begin{enumerate}
         \item $G$ a hyperbolic group and $d$ the word metric with respect to
           a finite generating set $S$.
         \item $G$ is finitely generated and hyperbolic relative to $\PP$,
           $S$ a finite relative generating set, and $d$ the word metric with
           respect to $S \cup \PP$.
         \item $G$ is the mapping class group $Mod(S)$ and $d$    
           a metric 
           that  is equivariantly quasi-isometric to the curve
           complex $CC(S)$.
         \item $G$ is $Out(F_n)$ and $d$  
           a metric 
           that is equivariantly quasi-isometric to the free factor
           complex $\FF_n$.
         \end{enumerate}
         Then (respectively)
         \begin{enumerate}
            \item if $H$ is quasiconvex, then $(G,\{H\})$ has graded
              relative hyperbolicity; conversely,  if $(G,\{H\})$ has
              geometric graded
              relative hyperbolicity then $H$ is
              quasiconvex.
            \item if $H$ is relatively quasiconvex then $(G,\{H\},d)$ has
              graded relative hyperbolicity; conversely,  if $(G,\{H\},d)$ has
              geometric graded relative hyperbolicity then $H$ is relatively quasiconvex.
            \item if $H$ is convex cocompact in $Mod(S)$ then 
              $(G,\{H\},d)$ has graded relative hyperbolicity;
              conversely, if $(G,\{H\},d)$ has geometric  graded relative
              hyperbolicity  and the action of
              $H$ on the curve complex is uniformly proper, 
              then $H$ is convex cocompact in $Mod(S)$.
            \item if $H$ is convex cocompact in $Out(F_n)$ then
              $(G,\{H\},d)$ has graded relative hyperbolicity;
              conversely, if $(G,\{H\},d)$ has geometric  graded relative
              hyperbolicity  and the action of
              $H$ on the free factor complex is uniformly proper, then
              $H$ is convex cocompact in $Out(F_n)$.
          \end{enumerate}
        \end{theorem}

\medskip

\noindent {\bf Structure of the paper:}\\ In Section 2,
we will review the notions of  hyperbolicity for metric
spaces relative to subsets. This will be related to the notion of hyperbolic embeddedness
\cite{DGO}. We will need to generalize the notion of  hyperbolic embeddedness in \cite{DGO} to one of
coarse hyperbolic embeddedness in order to accomplish this. We will also prove
results on the preservation of quasiconvexity under
electrification. We give two sets of proofs:  the first set of
proofs relies on assembling diverse pieces of literature on relative hyperbolicity, with several minor
adaptations. We also give  a more self-contained set of
proofs relying on asymptotic cones.

 In Section 3.1 and the preliminary discussion in Section 4, we give an account of two notions
of height: algebraic and geometric. The classical (algebraic) notion of height of a subgroup
concerns the number of conjugates that can have  infinite
intersection. The notion of geometric height is similar, but instead of
considering infinite intersection, we consider unbounded intersections
in a (not necessarily proper) word metric. This naturally leads us to dealing with intersections
in different contexts:

\begin{enumerate}
\item Intersections of conjugates of subgroups in a proper $(\Gamma, d)$ (the Cayley graph of the ambient
group with respect to a finite generating set). 
\item Intersections of metric thickenings of cosets in a not necessarily
proper $(\Gamma, d)$.
\end{enumerate}

The first is purely group theoretic (algebraic) and the last
geometric. 
Accordingly, we have two notions of height: algebraic and 
 geometric.  
In line with this, we investigate two notions of graded relative
hyperbolicity in Section 4 (cf. Definition \ref{ggrh}):
\begin{enumerate}
\item Graded relative hyperbolicity (algebraic)
\item Graded geometric relative hyperbolicity
\end{enumerate}

In the fourth section, we  also introduce and study a qi-intersection
property, a property that ensures that quasi-convexity is preserved under passage
to electrified spaces. The property
exists in both variants above. 

In the fifth and the sixth sections, we will prove our main results
relating height and  geometric graded relative hyperbolicity. On a first reading, the
reader is welcome to keep the simplest (algebraic or group-theoretic) notion in mind. 
To get a hang of where the paper is headed, we suggest that the reader take a first look at 
Sections 5 and 6, armed with Section 3.3 and the 
statements of  Proposition \ref{prop;hyp_qc_have_fgh},
Theorem  \ref{relht},
Theorem  \ref{alght-mcg0}, 
Theorem  \ref{alght-out}  and
Proposition \ref{prop;satisfiesqiip}. This, we hope,  will clarify our intent.

 \section{Relative Hyperbolicity, coarse hyperbolic
          embeddings}\label{rh}

        We shall clarify here what it means in this paper for a
        geodesic space $(X,d)$, to be hyperbolic relative to a family
        of subspaces $\YY=\{Y_i, i\in I \}$, or to cast it in another
        language, what it means for the family $\YY$ to be
        hyperbolically embedded in $(X,d)$.  There are slight
        differences from the more usual context of groups and
        subgroups (as in \cite{DGO}), but we will keep the descending
        compatibility (when these notions hold in the context of
        groups, they hold in the context of spaces).

        We begin by recalling relevant constructions.

        \subsection{Electrification by cones} \label{sec;elec}

          Given a metric space $(Y, d_Y)$, we will endow $Y\times [0,1]$
          with the following product metric: it is the largest metric
          that agrees with $d_Y$ on $Y\times \{0\}$,
          and  each $\{y\} \times [0,1]$ is endowed with a metric isometric to
          the segment $[0,1]$.

        \begin{defn} \cite{farb-relhyp} Let $(X,d)$ be a geodesic
          length space, and $\YY=\{Y_i, i\in I \}$ be a collection of
          subsets of $X$.  The \emph{ electrification }
          $(X^{el}_\YY, d^{el}_{\YY})$ of $(X,d)$ along $\YY$ is
          defined as the following coned-off space: \\
          $X^{el}_\YY = X\sqcup \{\bigsqcup_{i\in I} Y_i \times [0,1]
          \} /\sim$
          where $\sim$ denotes the identification of
          $ Y_i\times \{0\}$ with $Y_i\subset X$ for each $i$, and the
          identification of $Y_i\times \{1\}$ to a single cone point
          $v_i$ (dependent on $i$).

          The metric $d^{el}_{\YY}$ is defined as the path metric on
          $X^{el}_\YY$ for the natural quotient metric coming from the
          product metric on $Y_i \times [0,1]$ (defined as above).
        \end{defn}

        Let $Y_i\in \YY$. The {\bf angular metric} $\hat{d}_{Y_i}$ (or
        simply, $\hat{d}$, when there is no scope for confusion) on
        $Y_i$ is defined as follows:\\ For $y_1,y_2 \in Y_i$,
        $\hat{d}_{Y_i} (y_1,y_2)$ is the infimum of lengths of paths
        in $X^{el}_\YY$ joining $y_1$ to $y_2$ not passing through the
        vertex $v_i$. (We allow the angular metric to take on infinity as a  value).
        
        If $(X,d)$ is a metric space, and $Y$ is a subspace, we write
        $d|_Y$ the metric induced on $Y$. 

        \begin{defn}\label{he}
          Consider a geodesic metric space $(X,d)$ and a family of subsets
          $\YY =\{Y_i, i\in I \}$.  We will say that $\calY$ is {\bf coarsely
            hyperbolically embedded} in $(X,d)$, if there is a function
          $\psi: \bbR_+\to \bbR_+$ which is proper ({\it i.e.}
          $\lim_{+\infty} \psi(x) = +\infty$), and such that
          \begin{enumerate}
          \item the electrified space $X^{el}_\YY$ is hyperbolic,
          \item the angular metric at each $Y \in \YY$ in the cone-off is
            bounded from below by $\psi \circ d|_Y$.
          \end{enumerate}
        \end{defn}

        \begin{rmark}
          This notion originates from Osin's \cite{OsinIJAC}, and was
          developed further in \cite{DGO} in the context of groups, where one
          requires that each subset $Y_i \in \calY$ is a proper metric space
          for the angular metric. This automatically implies the weaker
          condition of the above definition. The converse is not true: if
          $\YY$ is a collection of uniformly bounded subgroups of a group $X$
          with a (not necessarily proper) word metric, it is always coarsely
          hyperbolically embedded, but it is hyperbolically embedded in the
          sense of \cite{DGO} only if it is finite.
        \end{rmark}

        As in the point of view of \cite{OsinIJAC}, we say that $(X,d)$ is {
          \bf strongly hyperbolic relative to} the collection $\calY$ (in the
        sense of spaces) if $\calY$ is coarsely hyperbolically embedded in
        $(X,d)$.

        As we described in the remark, unfortunately, it happens that some
        groups (with a Cayley graph metric) are  hyperbolic relative to some
        subgroups in the sense of spaces, but not in the sense of groups.

        Note that there are other definitions of relative hyperbolicity for
        spaces. Dru\c{t}u introduced the following definition: a metric space
        is hyperbolic relative to a collection of subspaces if all asymptotic
        cones are tree graded with pieces being ultratranslates of asymptotic
        cones of the subsets.

        \subsection{Quasiretractions}

        \begin{defn}
          If $(X,d)$ is a metric space, and $Y\subset X$ is a subset endowed
          with a metric $d_Y$, we say that $d_Y$ is $\lambda$-undistorted in
          $(X,d)$ if for all $y_1, y_2 \in Y$,
          $$\lambda^{-1} d(y_1,y_2) -\lambda  \leq d_Y(y_1, y_2) \leq \lambda
          d(y_1,y_2) +\lambda. $$

          We say that $d_Y$ is undistorted in $(X,d)$ if it is
          $\lambda$-undistorted in $(X,d)$ for some $\lambda$.
        \end{defn}

        For the next proposition, define the $D$-coarse path metric on a
        subset $Y$ of a path-metric space $(X,d)$ to be the metric on $Y$
        obtained by taking the infimum of lengths over paths for which any
        subsegment of length $D$ meets $Y$.

        The next Proposition is the translation (to the present context) of
        Theorem 4.31 in \cite{DGO}, with a similar proof. 

        \begin{prop}\label{retraction}
          Let $(X,d)$ be a graph. Assume that $\YY$ is coarsely hyperbolically
          embedded in $(X,d)$. Then, there exists $D_0$ such that for all
          $Y\in \YY$, the $D_0$-coarse path metric on $Y\in \YY$ is
          undistorted (or equivalently the $D_0$-coarse path metric on
          $Y\in \YY$ is quasi-isometric to the metric induced from $(X,d)$).
        \end{prop}

        We will need the following Lemma, which originates in Lemma 4.29 in
        \cite{DGO}. The proof is the same; for convenience we will briefly
        recall it. The lemma provides quasiretractions onto hyperbolically
        embedded subsets in a hyperbolic space.

        \begin{lemma} Let $(X, d)$ be a geodesic metric space. There exists
          $C>0$ such that whenever $\YY$ is coarsely hyperbolically embedded
          (in the sense of spaces) in $(X,d)$, then for each $Y\in \YY$, there
          exists a map $r: X\to Y$ which is the identity on $Y$ and such that
          $\hat{d} (r(x), r(y)) \leq C d(x,y)$.
        \end{lemma}

        \begin{proof}
          Let $p$ the cone point associated to $Y$, and for each $x$, choose a
          geodesic $[p,x]$ and define $r(x)$ to be the point of $[p,x]$ at
          distance $1$ from $p$. Then $r(x)$ is in $Y$, and to prove the
          lemma, one only needs to check that there is $C$ such that if
          $d(x,y)=1$, then $\hat{d} (r(x), r(y)) \leq C $.  The constant $C$
          will be $10(\delta+1) +1$.
          Assume that $x$ and $y$ are at distance $>5(\delta+1)$ from the cone
          point. By hyperbolicity, one can find two points in the triangle
          $(p, x, y)$ at distance $2(\delta+1)$ from $r(x), r(y)$ at distance
          $\leq 2\delta$ from each other. This provides a path of length at
          most $6\delta + 4$. Hence $\hat{d} (r(x), r(y)) \leq 6\delta + 4$.
          If $x$ and $y$ are at distance $\leq 5(\delta+1)$ from the cone
          point, then
          $\hat{d} (r(x), r(y)) \leq d(r(x),x) + d(x,y) + d(y,r(y))\leq
          2\times 5(\delta+1) +1$.
        \end{proof}

        We can now prove the Proposition.
        \begin{proof}
          Choose $D_0 = \psi(C) $: for all $y_0, y_1 \in Y$ at distance
          $\leq D_0$, their angular distance is at most $  C$ (where $C$ is as given by the
          Lemma above).  Consider any path in $X$ from $y_0$ to $y_1$, call the
          consecutive vertices $z_0, \dots, z_n$, and project that path by
          $r$.  One gets $r(z_0), \dots, r(z_n)$ in $Y$, two consecutive ones
          being at distance at most $D_0$. This proves the claim.
        \end{proof}

        \begin{cor}\label{cor-retraction}
          If $(X,d)$ is hyperbolic, and if $\YY$ is coarsely hyperbolically
          embedded, then there is $C$ such that any $Y\in \YY$ is
          $C$-quasiconvex in $X$.
        \end{cor}

        \subsection{Gluing horoballs}
          Given a metric space $(Y, d_Y)$ , one can construct several models of
          combinatorial horoballs over it. We recall a construction
          (similar to that of Groves and Manning \cite{GM} for a graph).

          We consider inductively on $k \in \mathbb{N}\setminus \{0\}$ the space
          $\HH_k(Y) = Y\times [1, k]$ with the maximal metric $d_k$
          that 
          \begin{itemize}
             \item induces an isometry of $\{y\} \times [k-1, k]$ with $[0,1]$ for
               all $y\in Y$, and all $k\geq 1$,
             \item is at most $d_{k-1}$ on $\HH_{k-1}(Y) \subset \HH_k(Y)$
             \item coincides with $2^{-k}\times d$ on $Y\times \{k\}$.
           \end{itemize}
Then $\HH(Y)$ is defined as the inductive limit of the $\HH_k(Y)$'s and is called the horoball
over $Y$.
           Let $(X,d)$ be a graph, and $\YY$ be a collection of subgraphs (with
           the induced metric on each of them).  The \emph{horoballification} of
           $(X,d)$ over $\YY$ is defined to be the space
           $X^h_\YY = X\sqcup \{\bigsqcup_{i\in I} \HH(Y_i)  \} /\sim_i$
           where $\sim_i$ denotes the identification of the boundary horosphere
           of $ \HH(Y_i) $ with $Y_i\subset X$.  The metric $d^h_{\YY}$ is
           defined as the path metric on $X^h_\YY$.

           One can electrify a horoballification $X^h_\calY$ of a space $(X,d)$:
           one gets a space quasi-isometric to the electrification $X^{el}_\YY$
           of $X$. We record this observation in the following.
           
           \begin{prop} \label{double-elec} 
             Let $X$ be a graph, and $\YY$ be a
             family of subgraphs. Let $X^{el}_\YY$ and $X^h_\YY$ be the
             electrification, and the horoballification as above.  Let
             $(X^{h})^{el}_{\HH(\YY)}$ be the electrification of $X^h_\YY$ over
             the collection of horoballs $\HH(Y_i)$ over $Y_i , i\in I$.

             Then there is a natural injective map
             $X^{el}_\YY \hookrightarrow (X^{h})^{el}_{\HH(\YY)}$ which is the
             identity on $X$ and sends the cone point of $Y_i$ to the cone
             point of $\HH(Y_i)$.

             Consider the map $e: ((X^{h})^{el}_{\HH(\YY)})^{(0)} \to X^{el}_\YY$
             that
             \begin{itemize}
                \item is the identity on $X$,
                \item sends each vertex of $\HH(Y_i)$ of depth $> 2$ to
                  $v_i \in X^{el}_\YY$,
                \item sends each vertex $(y,n) \in \HH(Y_i)$ of depth $n\leq 2$ to
                  $y \in Y_i \subset X$,
                \item sends the cone point of $((X^{h})^{el}_{\HH(\YY)})$ associated
                  to $\HH(Y_i)$ to the cone point of $X^{el}_\YY$ associated to
                  $Y_i$.
             \end{itemize}

             Then $e$ is a quasi-isometry that induces an isometry on
             $X^{el}_\YY \subset (X^{h})^{el}_{\HH(\YY)}$.
           \end{prop}

           \begin{proof}
             First, note that a geodesic in $(X^{h})^{el}_{\HH(\YY)}$ between two
             points of $X$ never contains an edge with a vertex of depth
             $\geq 1$. If it did, the subpath in the corresponding horoball would
             be either non-reduced, or would contain at least $3$ edges, and
             could be shortened by substituting a pair of edges through the cone
             attached to that horoball. Thus the image of such a geodesic under
             $e$ is a path of the same length.  In other words, there is an
             inequality on the metrics
             $d_{ X^{el}_\YY} \leq d_{(X^h)^{el}_{\HH(\YY)} }$ (restricted to
             $X^{el}_\YY$).  On the other hand, there is a natural inclusion
             $X^{el}_\YY \subset (X^{h})^{el}_{\HH(\YY)} $, and therefore on
             $X^{el}_\YY$, $ d_{(X^h)^{el}_{\HH(\YY)} } \leq d_{ X^{el}_\YY}$.
             Thus $e$ is an isometry on $X^{el}_\YY$. Also, every point in
             $ (X^{h})^{el}_{\HH(\YY)} $ is at distance at most $2$ from a point
             in $X$, hence also from a point of the image of $X^{el}_\YY$.
           \end{proof}

           \subsection{Relative hyperbolicity and hyperbolic embeddedness}

           Recall that we say that a subspace $Q$ of a geodesic metric space
           $(X,d)$ is $C$-quasiconvex, for some number $C>0$, if for any two
           points $x,y \in Q$, and any geodesic $[x,y]$ in $X$, any point of
           $[x,y]$ is at distance at most $C$ from a point of $Q$.

           \begin{defn} \cite[Definition 3.5]{mahan-ibdd}
             A collection
             $\mathcal{H}$ of (uniformly) $C$-quasiconvex sets in a
             $\delta$-hyperbolic metric space $X$ is said to be {\bf mutually
               D-cobounded} if for all $H_i, H_j \in \mathcal{H}$, with
             $H_i\neq H_j$, $\pi_i (H_j)$ has diameter less than $D$, where
             $\pi_i$ denotes a nearest point projection of $X$ onto $H_i$. A
             collection is {\bf mutually cobounded } if it is mutually
             D-cobounded for some $D$. \label{cobdd}
           \end{defn}

           The aim of this subsection is to establish criteria for hyperbolicity
           of certain spaces (electrification, horoballification), and related
           statements on persistence of quasi-convexity in these spaces.  We also
           show that hyperbolicity of the horoballification implies strong
           relative hyperbolicity, or coarse hyperbolic embeddedness.

           Two sets of arguments are given. In the first set of arguments, the
           pivotal statement is of the following form: Electrification or
           de-electrification preserves the property of being a quasigeodesic.
           The arguments are essentially existent in some form in the literature,
           and we merely sketch the proofs and refer the reader to specific
           points in the literature where these may be found.
           
           The second set of arguments uses asymptotic cones (hence the axiom of
           choice) and is more self-contained (it relies on
           Gromov-Cartan-Hadamard theorem). We decided to give both these
           arguments so as to leave it to the the reader to choose according to
           her/his taste.
           
           \subsection{Persistence of hyperbolicity and quasiconvexity}
           \subsubsection{The Statements}\label{thestatements}
           Here we state the results for which we give arguments in the following
           two subsubsections.
           
           \begin{prop}\label{prop;criterion}
             Let $(X,d)$ be a hyperbolic geodesic space, $C>0$, and $\YY$ be a
             family of $C$-quasiconvex subspaces.  Then $X^{el}_\YY$ is
             hyperbolic.  If moreover the elements of $\YY$ are mutually
             cobounded, then $X^h_\YY$ is hyperbolic.
           \end{prop}

           In the same spirit, we also record the following statement on
           persistence of quasi-convexity.
           
           \begin{prop}\label{cobpersists}
             Given $\delta , C$ there exists $ C'$ such that if
             $(X,d_X)$ is a $\delta$-hyperbolic metric space with a collection
             $\mathcal{Y}$ of $C$-quasiconvex, 
             sets. 
             then the following holds: \\
             If $Q (\subset X)$ is a $C$-quasiconvex set (not necessarily an
             element of $\YY$), then $Q$ is $C'$-quasiconvex in
             $(X^{el}_\YY , d_e)$.

           \end{prop}

           Finally, there is a partial converse. We need a little bit of
           vocabulary.
           
           If $Z$ is a subset of a metric space $(X,d)$, a $(d,R)$-coarse path in
           $Z$ is a sequence of points of $Z$ such that two consecutive points are
           always at distance at most $R$ for the metric $d$.

           Let $H$ and $Y$ two subsets of $X$. We will denote by $H^{+\lambda}$
           the set of points at distance at most $\lambda$ from $H$.
           We will say
           that $H$ $(\Delta,\epsilon)$-meets $Y$ if there are two points
           $x_1,x_2$  at distance  $\geq \Delta$ from each
           other, and at distance $\leq \epsilon \Delta$ from
           $Y$ and $H$, and if for all pair of points  at distance $20\delta$
           from $\{x_1,x_2\}$, either  $H$ or  $Y$ is at distance
            at least $\epsilon \Delta -2\delta$ from one of them. 

             The two points $x_1, x_2$ are called a pair of meeting points in
           $H$ (for $Y$). We shall say that a subset $H$ of $X$ is coarsely path connected
           if there exists $c \geq 0$ such that the 
           $c-$neighborhood $N_c(H)$ is path connected.

           \begin{prop}\label{prop;unfolding_qc}
             
             Let $(X,d)$ be hyperbolic, and let $\calY$ be a collection of
             uniformly quasiconvex subsets.  Let $H$ be a subset of $X$ that is
             coarsely path connected, and quasiconvex in the electrification
             $X_\calY^{el}$.

             Assume also that there exists $\epsilon\in (0,1)$, and $\Delta_0$
             such that for all $\Delta> \Delta_0$, wherever $H$
             $(\Delta,\epsilon)$-meets an item $Y$ in $\calY$, there is a path in
             $H^{+\epsilon \Delta}$ between the meeting points in $H$ that is
             uniformly a quasigeodesic in the metric $(X,d)$.
             
             Then $H$ is quasiconvex in $(X,d)$.
             
             The quasiconvexity constant of $H$ can be chosen to depend only on
             the constants involved for $(X,d), \calY, \Delta_0, \epsilon$, the
             coarse path connection constant, and the quasi-geodesic constant of
             the last assumption.

           \end{prop}

           \subsubsection{Electroambient quasigeodesics}
           
           We recall here the concept of electroambient quasigeodesics from
           \cite{mahan-ibdd,mahan-split}.
           
           Let $(X,d)$ be a metric space, and $\YY$ a collection of subspaces.
           If $\gamma$ is a path in $(X,d)$, or even in
           $(X^{el}_\YY, d^{el}_{\YY})$, one can define an {\bf elementary
             electrification} of $\gamma$ in $(X^{el}_\YY,
           d^{el}_{\YY})$ as follows: \\
           For $x_1, x_2$ in $\gamma$, both belonging to some $Y_i\in \YY$, and
           at distance $>1$, replace the arc of $\gamma$ between them by a pair
           of edges $(x_1, v_i) (v_i,x_2)$, where $v_i$ is the cone-point
           corresponding to $Y_i$.

           A {\bf complete electrification} of $\gamma$ is a path obtained after
           a sequence of elementary electrifications of subarcs, admitting no
           further elementary electrifications.

           One can {\it de-electrify} certain paths.  Given a path $\gamma$ in
           $(X^{el}_\YY, d^{el}_{\YY})$, a {\bf de-electrification} of $\gamma$
           is a path $\sigma$ in $(X,d)$ such that
           \begin{enumerate}
           \item $\gamma$ is a complete electrification of $\sigma$,
           \item $(\sigma \setminus \gamma) \cap Y_i$ is either empty or a
             geodesic in $Y_i$.
           \end{enumerate}
           
           A {\bf $(\lambda,\mu)$-de-electrification} of a path $\gamma$ in
           $(X^{el}_\YY, d^{el}_{\YY})$, is a path in $X$ such that
           \begin{enumerate}
           \item $\gamma$ is a complete electrification of $\sigma$,
           \item $(\sigma \setminus \gamma) \cap Y_i$ is either empty or a
             $(\lambda,\mu)$-quasigeodesic in $Y_i$.
           \end{enumerate}

           Observe that, given a path $\sigma$ in $X^{el}$, there might be
           several ways to de-electrify it, but these ways differ only in the
           choice of the geodesic (or the quasi-geodesic) in the family of
           subspaces $Y_i$ corresponding to the successive cone points $v_{Y_i}$
           on the path $\sigma$. It might also happen that there is no way of
           de-electrifying it, if the spaces in $\YY$ are not quasiconvex.

           We say that a path $\gamma$ in $(X,d)$ is an {\bf electro-ambient
             geodesic} if it is a de-electrification of a geodesic.

           We say that it is a {\bf $(\lambda, \mu)-$ electro-ambient
             quasigeodesic } if it is the $(\lambda, \mu)-$ de-electrification of
           a $(\lambda, \mu)-$quasigeodesic in $(X^{el}_\YY, d^{el}_{\YY})$.

           \medskip

           We begin by discussing Proposition \ref{prop;criterion}.

           \begin{proof}
             The first part is fairly well-known. In some other guise it appears
             in \cite[Proposition 7.4]{bowditch-relhyp} \cite[Proposition
             1]{Szcz} \cite{mahan-ibdd}. In the first two, the electrification by
             cones is replaced by collapses of subspaces (identifications to
             points) which of course requires that the subspaces to electrify are
             disjoint and separated.

             However this is only a technical assumption (as explicated in
             \cite{mahan-ibdd}). Indeed, by replacing (or augmenting) any
             $Y\in \bbY$ by $Y\times [0,D]$ glued along $Y\times \{0\}$, and
             replacing $\YY$ by the family $\{ Y\times \{D\}, Y\in \YY \}$, we
             achieve a $D$-separated quasiconvex family.
           \end{proof}

           Next, we discuss Proposition \ref{cobpersists}.
           \begin{proof}
             The proofs of Lemma 4.5 and Proposition 4.6 of \cite{farb-relhyp}, Proposition 4.3 and
             Theorem 5.3 of \cite{klarreich} (see also \cite{bowditch-relhyp})
             furnish Proposition \ref{cobpersists}. 
             
             The crucial ingredient in all these proofs is the fact
             that in a hyperbolic space, nearest point projections decrease distance exponentially. 
             Farb proves this in the setup of horoballs in complete simply connected manifolds of pinched
             negative curvature. Klarreich "coarsifies" this assertion by generalizing it to the context of
             hyperbolic metric spaces.             
             \end{proof}

           The rest of this (subsub)section is devoted to discussing Proposition
           \ref{prop;unfolding_qc}. Towards doing this, we will obtain an
           argument for showing a variant of the second point of Proposition
           \ref{prop;criterion}, namely that in a hyperbolic space $(X,d)$, a
           family $\YY$ of uniformly quasi convex subspaces that is mutually
           cobounded defines a strong relative hyperbolic structure on
           $(X,d)$. The second point of \ref{prop;criterion} as it is stated will
           be however proved in the next subsection.

           We shall have need for the following Lemma \cite[Lemma 3.9]{mahan-ibdd}
           (see also \cite[Proposition 4.3]{klarreich} \cite[Lemma
           2.5]{mahan-split}).
           
           \begin{lemma} 
             Suppose $(X,d)$ is $\delta$-hyperbolic. Let $\mathcal{H}$ be a
             collection of $C$-quasiconvex $D$-mutually cobounded subsets. Then for all $P \geq 1$,
             there
             exists $\epsilon_0 = \epsilon_0 (C, P, D, \delta )$ such that the
             following holds:

             Let $\beta$ be an electric $(P,P)$-quasigeodesic without backtracking
             (i.e. $\beta$ does not return to any $H_1 \in \HH$ after leaving it)
             and $\gamma$ a  geodesic in $(X,d)$, both joining $x, y$. 
             Then, given
             $\epsilon \geq \epsilon_0$
             there exists $D = D(P, \epsilon )$ such that \\
             \begin{enumerate}
             \item {\it Similar Intersection Patterns 1:} if precisely one of
               $\{ \beta , \gamma \}$ meets an $\epsilon$-neighborhood
               $N_\epsilon (H_1)$ of an electrified quasiconvex set
               $H_1 \in \mathcal{H}$, then the length (measured in the intrinsic
               path-metric on $N_\epsilon (H_1)$ ) from the entry point to the
               exit point is at most $D$.
             \item {\it Similar Intersection Patterns 2:} if both
               $\{ \beta , \gamma \}$ meet some $N_\epsilon (H_1)$ then the
               length (measured in the intrinsic path-metric on
               $N_\epsilon (H_1)$ ) from the entry point of $\beta$ to that of
               $\gamma$ is at most $D$; similarly for exit points.
             \end{enumerate}
             \label{farb2A}
           \end{lemma}

           Note that Lemma \ref{farb2A} above is quite general and does not
           require $X$ to be proper. The two properties occurring in Lemma \ref{farb2A}
           were introduced by Farb \cite{farb-relhyp} in the context of a group $G$ and a collection $\HH$
           of cosets of a subgroup $H$. The two together are termed `Bounded Coset Penetration'
           in \cite{farb-relhyp}.

           \begin{remark} In \cite{mahan-ibdd}, the extra hypothesis of
             separatedness was used. However, this is superfluous by the same
             remark on augmentations of elements of $\YY$ that we made in the
             beginning of the proof of Proposition \ref{prop;criterion}.  Lemma
             \ref{farb2A} may be stated equivalently as the following
             (compare with \ref{prop;from_horo_to_he} below) .\\
             \noindent If $X$ is a hyperbolic metric space and $\mathcal{H}$ a
             collection of uniformly quasiconvex mutually cobounded subsets, then
             $X$ is strongly hyperbolic relative to the collection $\mathcal{H}$.
           \end{remark}

           We give a slightly modified version of \cite[Lemma 3.15]{mahan-ibdd}
           below by using the equivalent hypothesis of strong relative
           hyperbolicity (i.e. Lemma \ref{farb2A}).

           \begin{lemma} Let $(X,d)$ be a $\delta-$hyperbolic metric space, and
             $\HH$ a family of subsets such that $X$ is strongly hyperbolic
             relative to $\HH$. Then for all $\lambda, \mu >0$, there exists
             $\lambda',\mu'$ such that any electro-ambient
             $(\lambda,\mu)$-quasi-geodesic is a $(\lambda',\mu')$-quasi-geodesic
             in $(X,d)$.  \label{ea-strong} \end{lemma}

           The proof of Lemma \ref{ea-strong} goes through mutatis mutandis for
           strongly relatively hyperbolic spaces as well, i.e. hyperbolicity of
           $X$ may be replaced by relative hyperbolicity in Lemma \ref{ea-strong}
           above. We state this explicitly below:

           \begin{cor}\label{ea-cor}
             Let $(X,d)$ be strongly relatively hyperbolic relative to a
             collection $\YY$ of path connected subsets.  Then, for all
             $\lambda, \mu >0$, there exists $\lambda',\mu'$ such that any
             electro-ambient $(\lambda,\mu)$-quasi-geodesic is a
             $(\lambda',\mu')$-quasi-geodesic in $(X,d)$.
           \end{cor}

           We include a brief sketch of the proof-idea following
           \cite{mahan-ibdd}. Let $\gamma$ be an electro-ambient quasigeodesic.
           By Definition, its electrification $\hat{\gamma}$ is a quasi-geodesic
           in $X^{el}_\YY$. Let $\sigma$ be the electric geodesic joining the
           end-points of $\gamma$.  Hence $\sigma$ and $\hat{\gamma}$ have
           similar intersection patters with the sets $Y_i$ \cite{farb-relhyp},
           i.e. they enter and leave any $Y_i$ at nearby points. It then suffices
           to show that an electro-ambient representative of $\sigma$ is in fact
           a quasigeodesic in $X$. A proof of this is last statement is given in
           \cite[Theorem 8.1]{ctm-locconn} in the context of horoballs in
           hyperbolic space (see also Lemmas 4.8, 4.9 and their proofs in
           \cite{farb-relhyp}). The same proof works after horoballification for
           an arbitrary relatively hyperbolic space. $\Box$
           
           \medskip

           The proof of Proposition \ref{prop;unfolding_qc} as stated will be given in the next subsubsection.
           We shall provide here a proof that suffices for the purposes of this paper. We assume, in addition
           to the hypothesis of the Proposition that there exists an integer $n>0$, and $ D_0 \geq 0$ 
           such that for all distinct $Y_1, \dots, Y_n
           \in \YY$, $\cap_i Y_i^{+\epsilon}$ has diameter at most $D$. The existence of such a 
           number $n$ will translate into 
           the notion of finite geometric height later in the paper. 

           \begin{proof} We prove the statement by inducting on $n$. For $n=1$ there is nothing to
           	show; so we start with $n=2$. Note that in this case, the hypothesis is equivalent
           	to the assumption that the $Y_i$'s are cobounded.
             Assume therefore that the elements of $\YY$ are uniformly quasiconvex in
             $(X,d)$; and that they are uniformly mutually cobounded.
             We shall show that $H$ is also quasiconvex for a uniform
             constant.

             First, since $(X,d)$ is
             hyperbolic 
             it follows by Proposition \ref{prop;criterion} 
             that $(X^{el}_\YY,d^{el})$ is hyperbolic.

             Let $x, y \in H$. By
             assumption, there exists $C_0 \geq 0$ such that $H$ is
             $(C_0,C_0)-$qi embedded in $(G, d^{el})$. Denote by $\PP$ the set of
             cone points corresponding to elements of ${\YY}$ and let $\gamma$ be
             a $(C,C)-$quasi-geodesic without backtracking in $(X,d^{el})$ with
             vertices in $H \cup \PP$ joining $x,y \in H$.  By assumption, the
             collection ${\YY}$ is uniformly $C$-quasiconvex.  Further, by
             assumption, there exists $\epsilon\in (0,1)$, and $\Delta_0$ such
             that for all $\Delta> \Delta_0$, wherever $H$
             $(\Delta,\epsilon)$-meets an item $Y$ in $\calY$, there is a path in
             $H^{+\epsilon \Delta}$ between the meeting points in $H$ that is
             uniformly a quasigeodesic in the metric $(X,d)$.  Hence, for some
             uniform constants $\lambda, \mu$, we may (coarsely)
             $(\lambda,\mu)$-de-electrify $\gamma$ to obtain a
             $(\lambda,\mu)$-electro-ambient quasigeodesic $\gamma'$ in $(X,d)$,
             that lies close to $H$.
             [Note that the meeting points of $H$ with elements of $\YY $ are
             only coarsely defined. So we are actually replacing pieces of
             $\gamma$ by quasigeodesics in $H^{+\epsilon \Delta}$ rather than in
             $H$ itself.]

             Note that by assumption $\YY$ are uniformly quasiconvex in
             $(X,d)$; and further that they are uniformly mutually cobounded. Hence
             the space $X$ is actually strongly  hyperbolic
             relative to $\YY$. 
             By Corollary \ref{ea-cor}, it follows that $\gamma'$ is a
             quasi-geodesic in $(X, d)$, for a uniform constant.

             Since this was done for arbitrary $x,y \in H_{i,\ell}$, we obtain
             that $H $ is $D-$quasiconvex in $(X,d)$. This finishes the proof of
             Proposition \ref{prop;unfolding_qc} for $n=2$.
             
             The induction step is now easy. Assume that the statement is true for $n=m$. We shall prove 
             it for $n=m+1$. Electrify all pairwise intersections of $Y_i^{+\epsilon}$ to obtain an electric metric
             $d_2$. Then the collection $\{ Y_i\}$ is cobounded with
             respect to the electric metric $d_2$.   
             Here again, the space
             $(X,d_2)$ is strongly  hyperbolic relative to
             the collection $\{ Y_i\}$.  By
             the argument in the case $n=2$ above, $H$ is quasiconvex in $(X,d_2)$. The collection of
             pairwise intersections of the $Y_i^{+\epsilon}$'s in $X$ satisfies the property that an intersection
             of any $m$ of them is bounded. We are then through by the induction hypothesis.
           \end{proof}

           \subsubsection{Proofs through asymptotic cones}
           The repeated use of different references coming from different
           contexts in the previous subsubsection might call for more systematic
           self-contained proofs of the statements of subsection
           \ref{thestatements}. This is our purpose in this subsubsection.

           In this part we will use the structure of an argument originally due to
           Gromov, and developed by Coulon amongst others (see for instance
           \cite[Proposition 5.28]{Coulon_IJAC}), which uses asymptotic cones in
           order to show hyperbolicity or quasiconvexity of constructions.

           We fix a non-principal ultrafilter $\omega$
           and will use the
           construction of asymptotic cones with respect to this ultrafilter $\omega$. A few
           observations are in order here.

           In all the following, $(X_N,x_N)$ is a sequence of pointed
           $\delta_N$-hyperbolic spaces, with $\delta_N$ converging to
           $0$. Recall  then that 
           the asymptotic cone $\lim_\omega (X_N,x_N)$ is an $\mathbb{R}$-tree,
           with a base point.

           If $\YY_N$ is a family of $c_N$-quasiconvex subsets of
           $X_N$ (for $c_N$ tending to $0$), we want to
           consider the asymptotic cone $\lim_\omega
           ((X_N)_{\YY_N}^{el}, x_N)$ and relate it to $\lim_\omega
           (X_N,x_N)$.    

          Let us define the following equivalence relation on the set of
          sequences in $X_N$.   Two sequences $(u_N), (v_N)$ are
          equivalent if $d_{X_N} (u_N,v_N) = O(1)$ for the
          ultrafilter $\omega$ (more precisely, if there exists a
          constant $C$ such that for $\omega$-almost all values of
          $N$, $d_{X_N} (u_N,v_N) \leq C$).  Let us consider
          the set of equivalence classes of sequences, and let us only
          keep those that have some (hence all) representative $(u_N)$
          such that the electric distance $d_{X_N^{el}}(x_N, u_N) $ is
          $O(1)$. Let us call $\frakC$ this set of equivalence
          classes.  For a sequence $u=(u_N)$ we write $\sim_u$ for its class
          in $\frakC$.  We also allow ourselves to write   $\lim_\omega (X_N,
        \sim_u)$ for $\lim_\omega (X_N,        u_N)$ to avoid cluttered notation.  

          \begin{lemma} \label{lem;manycopies}
            There is a natural inclusion from the disjoint union
            $\displaystyle \bigsqcup_{\sim_u \in \frakC}  \lim_\omega (X_N, \sim_u)$
              into $\lim_\omega     ((X_N)_{\YY_N}^{el}, x_N)$.
          \end{lemma}
\begin{proof}
By definition of $\frakC$ we have a well defined map
          $\lim_\omega (X_N, \sim_u) \to \lim_\omega
          ((X_N)_{\YY_N}^{el}, x_N)$ for each class $\sim_u$ in $\frakC$. 
          Given two sequences $(y_N)$, and $(z_N)$  in the same class
          $\sim_u\in \frakC$, if $d_{X_N}(z_N,y_N)$ is not  $o(1)$
          for $\omega$, then in the electric metric, it is not
          $o(1)$, since the added edges all have  length $1$.  Thus
          this map is injective. If  $(y_N)$, and $(z_N)$ are not in
          the same class in $\frakC$, then $d_{X_N}(z_N,y_N)$ is not
          $o(1)$ for $\omega$, and again, in the electric metric, it is not
          $o(1)$. The map of the lemma is thus injective.  
\end{proof}

          Note that the inclusion is continuous, as inclusions along
          the sequence are distance non-increasing.  But it is not isometric. We need to
          describe what happens with the cone off.
          
          Consider a sequence $(Y_N)$ of subsets of $\calY$. One says
          that the sequence is visible in  $\lim_\omega (X_N, u_N)$
          if $d_{X_N} (u_N, Y_N) \leq O(1)$. In that case,
          $\lim_\omega (Y_N, u_N)$ is a subset of  $\lim_\omega (X_N,
          u_N)$, consisting of the images of all the sequences of elements of $Y_N$
          that remain at $O(1)$-distance from $u_N$.  Note that given
          a sequence $(Y_N)$, it can be visible in several limits
          $\lim_\omega (X_N, u_N)$ (for several non-equivalent
          $(u_N)$).    In those classes where $(Y_N)$ is not visible,
          $\lim_\omega (Y_N,
          u_N) $ is empty. 
           Let us  define   $\lim_\omega (Y_N,  *)$
          to be  $\sqcup_{\sim_u \in \frakC}  \lim_\omega (Y_N, u_N)$
          in $\sqcup_{\sim_u \in \frakC}  \lim_\omega (X_N, u_N)$. By
          the previous lemma, this is naturally a subset of
          $\lim_\omega((X_N)^{\it el}_{\calY_N}, x_N)$.

          We define $\calY^\omega$ to be  the collection of all sets
          $\lim_\omega (Y_N,  *)$, for all possible sequences $(Y_N)
          \in \prod_{N>0} Y_N$.  This is  a family of subsets  of
          $\lim_\omega((X_N)^{\it el}_{\calY_N}, x_N)$.

          Let us define $\left[\sqcup_{\sim_u \in \frakC}  \lim_\omega (X_N,
          u_N)\right]^{el}$ to be the cone-off (of parameter $1$) of
        $\left[\sqcup_{\sim_u \in \frakC}  \lim_\omega (X_N,
          \sim_u)\right]$ over each $\lim_\omega (Y_N,  *)$.  Note
        that there is a natural copy of   $\left[\lim_\omega (X_N,
          \sim_u)\right]^{el}$  in   $\left[\sqcup_{\sim_u \in \frakC}  \lim_\omega (X_N,
          \sim_u)\right]^{el}$.

        \begin{lemma} \label{lem;electrification_of_ascone}
$\lim_\omega     ((X_N)_{\YY_N}^{el}, x_N)$ is isometric to $\left[\sqcup_{\sim_u \in \frakC}  \lim_\omega (X_N,          \sim_u)\right]^{el}$. 
          \end{lemma}
\begin{proof}

First there is a natural bijection from $\lim_\omega
((X_N)_{\YY_N}^{el}, x_N)$ to  $\left[\sqcup_{\sim_u \in \frakC}
  \lim_\omega (X_N,          \sim_u)\right]^{el}$.  Indeed, for any
sequence $(y_N)$  with $y_N\in X_N$ at distance $O(1)$ from $x_N$ in
the electric metric, Lemma \ref{lem;manycopies} provides an image  in
$\left[\sqcup_{\sim_u \in \frakC}  \lim_\omega (X_N,
  \sim_u)\right]^{el}$. For any sequence $(c_N)$  with $c_N$ a cone-point in
$X_N^{\it el}\setminus X_N$,  at distance $O(1)$ from $x_N$ in
the electric metric, $c_N$ is in the cone electrifying a certain
$Y_N$, which is therefore at distance $O(1)$ from $x_N$ for the
electric metric. Choose $u_N\in Y_N$, then the equivalence class of
$(u_N)$ is in $\frakC$, and of course $Y_N$ is visible in this
class. Thus, there is a cone point $c \in \left[\sqcup_{\sim_u \in
    \frakC}  \lim_\omega (X_N,          \sim_u)\right]^{el}$ at distance
$1$ from $\lim_\omega(Y_N, *)$. We choose this point as the image of
the sequence $(c_N)$ in $\lim_\omega     ((X_N)_{\YY_N}^{el},
x_N)$. This is well defined, and injective, for if $c'_N$ is another
sequence of cone-points $\omega$-almost everywhere different from
$c_N$, then it defines an $\omega$-almost everywhere different
sequence $Y'_N$, and a different set   $\lim_\omega(Y'_N, *)$. We also
can extend our map to all $\lim_\omega     ((X_N)_{\YY_N}^{el}, x_N)$
linearly on the cone-edges. This produces a bijection  $\lim_\omega
((X_N)_{\YY_N}^{el}, x_N)   \to  \left[\sqcup_{\sim_u \in \frakC}
  \lim_\omega (X_N,          \sim_u)\right]^{el}$.

To show that it is an isometry, consider two sequences $(y_N), (z_N)$
 both in $X_N$,
such that the distance in $X_{\YY_N}^{el}$ converges (for $\omega$) to
$\ell$. Then there is a path of length $\ell_N$ (converging to $\ell$)
in $X_{\YY_N}^{el}$ with, eventually, at
most $\ell/2$ cone points on it. It follows from the construction that
their images in   $\left[\sqcup_{\sim_u \in \frakC}
  \lim_\omega (X_N,          \sim_u)\right]^{el}$ are at distance at most
$\ell$. 
Conversely, assume  $(y_N)$ and $(z_N)$ are sequences
in $X_N$
giving points 
in  $\lim_\omega (X_N,          \sim_u)$ and  $\lim_\omega (X_N,
\sim_v)$, for $\sim_u, \sim_v \in\frakC$,  and take a path $\gamma$ between
these points in  $\lim_\omega
((X_N)_{\YY_N}^{el}, x_N)   \simeq \left[\sqcup_{\sim_u \in \frakC}
  \lim_\omega (X_N,          \sim_u)\right]^{el}$, of length $\ell >0$. It has finite
length, so it contains
finitely many cone points $c_i$ ($i=1, \dots, r$), coning $\lim_\omega
(Y_N^{(i)}, *)$, 
 for which  
$Y_N^{(i)}$ is visible in both  $\sim_{u_i}, \sim_{u_{i+1}}$.  This  easily
produces a path in $\lim_\omega     ((X_N)_{\YY_N}^{el}, x_N)$ of
length $\ell$, by using the corresponding cone points and the path
between the spaces $Y_N^{(i)}$ given by the restriction of $\gamma$.  

 We have thus observed that the bijection we started with is $1$-Lipschitz as is its inverse. It is therefore an isometry.
\end{proof}

We finally describe a tree-like structure on $\left[\sqcup_\frakC \lim_\omega (X_N,
  \sim_{u}) \right]^{el}$ where the pieces are the subspaces
$\left[\lim_\omega (X_N,          \sim_u)\right]^{el}$ for $\sim_u
\in \frakC$, which are
electrifications of $\lim_\omega (X_N,          \sim_u)$ over the
subsets of the form $\lim_\omega (Y_N, \sim_u)$, for all sequences
$(Y_N)\in \prod(\calY_N)$  that are visible in
$\sim_u$.       

        Let us say that two classes $\sim_u$ and $\sim_v$ are joined
        by a sequence $(Y_N)$ if the latter sequence is visible in both of
        them.  

First we describe a simpler case of this tree-like structure.

\begin{lemma} Assume that $\YY_N$ consists of $c_N$-quasiconvex
  subsets of $X_N$ with $c_N$ tending to $0$. 

For any pair of different classes $\sim_u, \sim_v$ in
  $\frakC$, the subspaces
 $\left[\lim_\omega (X_N,          \sim_u)\right]^{el}$  and
 $\left[\lim_\omega (X_N,          \sim_v)\right]^{el}$ of
 $\left[\sqcup_{\sim_w\in \frakC} \lim_\omega (X_N,
  \sim_{w}) \right]^{el}$  have
 intersection of diameter at most $2$.
\end{lemma}

\begin{proof}
Note that by Lemma \ref{lem;manycopies} the intersection consists of cone
points. Thus consider two sequences $(Y_N), (Y'_N)$ both visible in
$\sim_u$ and $ \sim_v$.   Consider then $y_N(u), y_N(v) \in Y_N$ such
that $y_N(u)$ is visible in $\sim_u$ (hence not in $\sim_v$) and symmetrically
$y_N(v)$ is visible in $\sim_v$ (hence not in $\sim_u$), and take $y'_N(u),
y'_N(v) \in Y'_N$ similarly. The distances $d_{X_N}(y_N(u), y'_N(u))$ and
$d_{X_N}(y_N(v), y'_N(v))$ both are $O(1)$ whereas  $d_{X_N}(y_N(u), y_N(v))$ and
$d_{X_N}(y'_N(u), y'_N(v))$ both go to infinity (for $\omega$). The
space $X_N$ being a $\delta_N$-hyperbolic space (for $\delta_N\to 0$),
the quadrilateral with these four vertices have their sides $[y_N(u),
y_N(v)]$ and $[y'_N(u), y'_N(v)]$ getting $o(1)$-close to each
other, on sequences that are visible for $\sim_u$ and sequences that
are visible for $\sim_v$. But these sides are close to $Y_N$ and
$Y'_N$ respectively.  It follows that in $\lim_\omega (X_N,
\sim_u)$ and in $\lim_\omega (X_N,          \sim_v)$, the limit of
$Y_N$ and of $Y'_N$ share a point. Thus the cone point of their
electrifications are at distance $2$.
\end{proof}

Note that if there is a bound on the diameter of the projection of
$Y_N$ on $Y'_N$, then there is only one point in the intersection.

\begin{lemma}
 If there is a cycle of classes $\sim_{u_1}, \sim_{u_2}, \dots,
 \sim_{u_k}, \sim_{u_{k+1}}=\sim_{u_1}$ where $\sim_{u_i}$ is joined
 to $\sim_{u_{i+1}}$ by a sequence $(Y^{(i)}_N)$, then there is
 $1<i_0<k+1$ such that   $(Y^{(i_0)}_N)$ is visible in $\sim_{u_1},
 \sim_{u_2}$ and the cone points of $(Y^{(i_0)}_N)$ and of
 $(Y^{(1)}_N)$ are at distance $2$
 from each other in $\left[\lim_\omega (X_N,
   \sim_{u_1})\right]^{el}$, and in  $\left[ \sqcup_\frakC \lim_\omega (X_N,
   \sim_u)\right]^{el}$.

\end{lemma}
          
As a corollary, $\lim_\omega     ((X_N)_{\YY_N}^{el}, x_N)$ is a
quasi-tree of the spaces  $\left[\lim_\omega (X_N,
  \sim)\right]^{el}$, and more precisely, all paths from   $\left[\lim_\omega (X_N,
  \sim_u)\right]^{el}$ to  $\left[\lim_\omega (X_N,
  \sim_v)\right]^{el}$, if $\sim_u\neq \sim_v$ have to pass through
the $2$-neighborhood of a certain cone point of $\left[\lim_\omega (X_N,
  \sim_u)\right]^{el}$.

\begin{proof}
The number $k$ is fixed, and the argument will generalize the one of the previous lemma. For each $i$,
let $(y_N^{(i)})$  and $(z_N^{(i)})$  be  sequences of points of
$Y_N^{(i)}$ respectively visible in
$\sim_{u_i}$  and in  $\sim_{u_{i+1}}$. One has $d_{X_N}(y_N^{(i)},
z_N^{(i)})$ unbounded, and    $d_{X_N}(z_N^{(i)},
y_N^{(i+1)}) = O(1)$. Therefore in the $2k$-gon $(y_N^{(1)}, z_N^{(1)},
y_N^{(2)}, z_N^{(2)}, \dots, y_N^{(k)}, z_N^{(k)}   )$, using the
approximation by a finite tree (for hyperbolic spaces), we see that  one of the
segments $[ y_N^{(i)},
z_N^{(i)} ]$ $(i\neq 1)$  must come $k\delta_N$-close to $[y_N^{(1)},
z_N^{(1)} ]$, and at distance $O(1)$ from $y_N^{(1)}$.

 After  extracting a subsequence, one can assume that $i$ is constant in $N$, and we
 choose it to be  our 
$i_0$.  It follows that the sequence $(Y_N^{(i_0)})$ is visible for
$\sim_{u_1}$ and the limit of $(Y_N^{(i_0)})$ and of $(Y_N^{1})$ share a
point in  $\lim_\omega (X_N,          \sim_{u_1})$. The conclusion
that the cone points of $(Y^{(i_0)}_N)$ and of
 $(Y^{(1)}_N)$ are at distance $2$
 from each other in $\left[\lim_\omega (X_N,
   \sim_{u_1})\right]^{el}$ follows, and this also implies that they are
 at distance at most $2$ in   $\left[ \sqcup_\frakC \lim_\omega (X_N,
   \sim_u)\right]^{el}$.

\end{proof}

Note that  if the subsets of
           $\YY_N$ are $c_n$-mutually cobounded, with $c_n$ going to
           $0$ (or even bounded) then one can improve the lemma by
           saying that  eventually $Y_N^{(i_0)} = Y_N^{(1)}$.

From the previous lemmas we get:
\begin{coro} \label{coro;quasitree}

$\left[\sqcup_\frakC \lim_\omega (X_N,
  \sim_{u}) \right]^{el}$  is the union of spaces of the form  $\left[\lim_\omega (X_N,          \sim_u)\right]^{el}$ for $\sim_u
\in \frakC$, with some cone points
identified.

Moreover, if  $\sim_u \neq \sim_v$,   and if $\gamma_1, \gamma_2$  are
any (finite)  paths from  $\left[\lim_\omega (X_N,
  \sim_u)\right]^{el}$ to $\left[\lim_\omega (X_N,
  \sim_v)\right]^{el}$, then for each $i\in \{1,2\}$, there exists a cone point $c_i \in
\left[\lim_\omega (X_N,          \sim_u)\right]^{el}$ in $\gamma_i$
such that the distance between $c_1$ and $c_2$ is at most $2$. 
  
\end{coro}

Indeed, such a pair of paths provides us with a certain finite cycle of
classes, starting with $\sim_u$, and we may apply the previous lemma.

We say that  $\left[\sqcup_\frakC \lim_\omega (X_N,
  \sim_{u}) \right]^{el}$ is a $2$-quasi-tree of spaces of the form  
$\left[\lim_\omega (X_N,          \sim_u)\right]^{el}$ for $\sim_u
\in \frakC$.

           Let us prove Proposition \ref{prop;criterion}. For  convenience
           of the reader we repeat the statement.

           \begin{prop*}\ref{prop;criterion} 
             Let $(X,d)$ be a hyperbolic geodesic space, $C>0$, and $\YY$ be a
             family of $C$-quasiconvex subspaces.  Then $X^{el}_\YY$ is
             hyperbolic.  If moreover the elements of $\YY$ are mutually
             cobounded, then $X^h_\YY$ is hyperbolic.
           
           \end{prop*}

           \begin{proof}
             We claim that, for all $\rho$, there exists $\delta_0<\rho/10^{14}$
             and $C_0<\rho/10^{14}$ such that if $X$ is $\delta_0$-hyperbolic,
             and if $\YY$ is a collection of $C_0$-quasiconvex subsets, then
             every ball of radius $\rho$ of $X^{el}_\YY$ is $10$-hyperbolic.

             For proving the claim, assume it false, and consider a sequence of
             counterexamples $(X_N, \YY_N)$ for $\delta_0= C_0=\frac{1}{N}$,
             $N = 1, 2, \dots$.  This means that $(X_N)_{\YY_N}^{el}$ fails to be
             $10$-hyperbolic. There are four points $x_N,y_N, z_N, t_N$, all at
             distance at most $2\rho$ from $x_N$, such that
             $(x_N, z_N)_{t_N} \leq \inf \{ (x_N, y_N)_{t_N} , (y_N, z_N)_{t_N}
             \}-10$.
             We pass to the ultralimit for $\omega$. In $\lim_\omega (X_N^{el}, x_N)$,
             each sequence $x_N, y_N, z_N, t_N$ converges, since these points
             stay at bounded distance from $x_N$, and the inequality
             persists. Hence one gets four points falsifying the
             $10$-hyperbolicity condition in a pointed space
             $\lim_\omega ((X_N)_{\YY_N}^{el}, x_N)$.

             But the asymptotic cone $\lim_\omega ((X_N)_{\YY_N}^{el},
             x_N)$ is, by Corollary \ref{coro;quasitree}, a $2$-quasi-tree
             of spaces that are  electrifications (of parameter $1$) of  real
             trees $\lim_\omega (X_N,x'_N)$, for some base point
             $x'_N$  
            over the family $\YY^\omega$,
             consisting of \emph{convex} subsets ({\it i.e.} of
             subtrees). 

             This space $ (\lim_\omega (X_N,x'_N))^{el}_{\YY^\omega }$ has
             $2$-thin geodesic triangles, therefore   $\lim_\omega ((X_N)_{\YY_N}^{el},
             x_N)$      itself is $10$-hyperbolic, a
             contradiction. The claim hence holds: $X^{el}_\YY$ is $\rho$-locally
             $10$-hyperbolic.

             We now claim that, under the same hypothesis, it is
             $(2+10C_0+10\delta_0)$-coarsely simply-connected, that is to say
             that any loop in it can be homotoped to a point by a sequence of
             substitutions of arcs of length $< (2+10C_0+10\delta_0)$ by its
             complement in a loop of length $< (2+10C_0+10\delta_0)$. Indeed, any
             time such a loop passes through a cone point associated to some
             $Y\in \YY$, one can consider a geodesic in $X$ between its entering
             and exiting points in $Y$, which stays in the $C_0$ neighborhood of
             $Y$. Therefore, a $(2+10C_0)$-coarse homotopy of the loop transforms
             it into a loop in $X$, which is $\delta$-hyperbolic. Since a
             $\delta$- hyperbolic space is $10\delta$-coarsely simply-connected,
             the second claim follows.

             The final ingredient is the Gromov-Cartan-Hadamard theorem \cite[Theorem
             A.1]{Coulon_IJAC}, stating that, if $\rho$ is sufficiently large
             compared to $\mu$, any $\rho$-locally $10$-hyperbolic space which is
             $\mu$-coarsely simply connected is (globally) $\delta'$-hyperbolic,
             for some $\delta'$.  We thus get that there exists
             $\delta_0<\rho/10^{14}$ and $C_0<\rho/10^{14}$ such that if $X$ is
             $\delta_0$-hyperbolic, and if $\YY$ is a collection of $C_0$-quasiconvex subsets, then $X^{el}_\YY$ is $\delta'$-hyperbolic.

             Now let us argue that this implies the first point of the
             proposition. If $X$ and $\YY$ are given as in the statement, one may
             rescale $X$ by a certain factor $\lambda>1$, so that it is
             $\delta_0$-hyperbolic, and such that $\YY$ is a collection of
             $C_0$-quasiconvex subsets. Let us define $X_\YY^{el_{\lambda}}$ to
             be
             $$X^{el_{\lambda}}_\YY = X\sqcup \{\bigsqcup_{i\in I} Y_i \times
             [0,\lambda] \} /\sim$$
             where $\sim$ denotes the identification of $ Y_i\times \{0\}$ with
             $Y_i\subset X$ for each $i$, and the identification of
             $Y_i\times \{1\}$ to a single cone point $v_i$ (dependent on $i$),
             and where $Y_i \times [0,\lambda]$ is endowed with the product
             metric as defined in the first paragraph of \ref{sec;elec} except
             that $\{y\} \times [0,n]$ is isometric to $[0,\lambda]$.  The claim
             ensures that $X_\YY^{el_{\lambda}}$ is hyperbolic. However, it is
             obviously quasi-isometric to $X_\YY^{el}$. We have the first point.

             For the second part, one can proceed with a similar proof, with
             horoballs.

             The claim is then that for all $\rho$, there exist $\delta_0, C_0$
             and $D_0$ such that if $X$ is $\delta_0$-hyperbolic, and if $\YY$ is
             a collection of $C_0$-quasiconvex subsets, $D_0$-mutually
             cobounded, then any ball of radius $\rho$ of the horoballification
             $X^h_\YY$ is $10$-hyperbolic.

             The proof of the claim is similar. Consider a sequence of
             counterexamples $X_N, \YY_N$, for the parameters
             $\delta = C = D = 1/N$ for $N$ going to infinity, with the four
             points $x_N, y_N, z_N, t_N$ in $(X_N)^h_\YY$, in a ball of radius
             $\rho$, falsifying the hyperbolicity condition.

             There are two cases. Either $x_N$ (which is in $(X_N)^h_\YY$)
             escapes from $X_N$, i.e. its distance from some basepoint in $X_N$
             tends to $\infty$ for the ultrafilter $\omega$, or it does not. In
             the case that it escapes from $X_N$, then, when it is larger than $\rho$
             all four points $x_N, y_N, z_N, t_N$ are in a single horoball, but
             such a horoball is $10$-hyperbolic hence a contradiction.

             The other case is  when there is $x'_N\in X_N$ whose
             distance to $x_N$ remains bounded (for the ultrafilter $\omega$).
             Note that $\{ x_N, y_N, z_N, t_N\}$ converge in the asymptotic cone
             $\lim_\omega ((X_N)_{\YY_N}^h, x'_N )$ of the sequence of pointed
             spaces $( (X_N)_{\YY_N}^h, x'_N)$.  It is also immediate by
             definition of $\lim_\omega \YY_N$ that
             $\lim_\omega ((X_N)_{\YY_N}^h, x'_N )$ is the horoballification of
             the asymptotic cone of the sequence $(X_N, x'_N)$ over the family
             $\lim_\omega (\YY_N,x'_N)$ defined above.

             This family $\lim_\omega (\YY_N, x'_N)$ consists of \emph{convex}
             subsets (hence subtrees), such that any two share at most one point.
             This horoballification is therefore a tree-graded space in the sense
             of \cite{DrutuSapir}, with pieces being the combinatorial horoballs
             over the subtrees constituting $\lim_\omega \YY_N$.  As a tree of
             $10$-hyperbolic spaces, this space is $10$-hyperbolic, contradicting
             the inequalities satisfied by the limits
             $\lim_\omega \{ x_N, y_N, z_N, t_N\}$.  Therefore, $X_\YY^h$ is
             $\rho$-locally $10$-hyperbolic.

             As before, one may check that (under the same assumptions) $X^h_\YY$
             is $(2+10C_0+10\delta_0)$-coarsely simply connected, and again this
             implies by the Gromov-Cartan-Hadamard theorem that (under the same
             assumptions) $X^h_\YY$ is hyperbolic.

             This implies the second point. Indeed, let us denote by
             $\frac{1}{\lambda} X$ the space $X$ with metric rescaled by
             $ \frac{1}{\lambda}$.

             The previous claim shows that, under the assumption of the second
             point of the proposition, there exists $\lambda>1$ such that
             $\lambda (\frac{1}{\lambda} X)^{h}_{\frac{1}{\lambda}\YY} $ is
             hyperbolic. Consider the map $\eta$ between
             $ X^{h}_\YY \to \lambda (\frac{1}{\lambda}
             X)^{h}_{\frac{1}{\lambda}\YY} $
             that is identity on $X$ and that sends $\{y\} \times \{n\}$ to
             $\{y\}\times \{ \lambda \times ( n + \lfloor \log_2 \lambda \rfloor
             ) \}$
             for all $y\in Y_i$ and all $Y_i$ (and all $n$).  All paths in
             $X_\YY^h$ that have only vertical segments in horoballs have their
             length expanded (under the map $\eta$) by a factor  between $1$
             and $\lambda+\log_2\lambda$.  But the geodesics in $X_\YY^h$ and
             $ \lambda (\frac{1}{\lambda} X)^{h}_{\frac{1}{\lambda}\YY} $ are
             paths whose components in horoballs consist of a vertical
             (descending) segment, followed by a single edge, followed by a
             vertical ascending segment (see \cite{GM}).
             Hence 
             $\eta$ is a quasi-isometry, and the space $ X^{h}_\YY $ is
             hyperbolic. 
           \end{proof}

           We continue with the persistence of quasi-convexity.

           \begin{prop*}\ref{cobpersists}  Given $\delta , C$ there exists
             $C'$ such that if $(X,d_X)$ is a $\delta$-hyperbolic metric
             space with a collection $\mathcal{Y}$ of $C$-quasiconvex, 
             sets, then the following holds: \\
             If $Q (\subset X)$ is some (any) $C$-quasiconvex set (not
             necessarily an element of $\YY$), then $Q$ is $C'$-quasiconvex in
             $(X^{el}_\YY , d_e)$.

           \end{prop*}

           \begin{proof}
             The strategy is similar to that in the previous proposition. The
             main claim is that for all $\rho$, there is $\delta_0<1, C_0<1$ such
             that if $(X,d_X)$ is $\delta_0$-hyperbolic, if $\mathcal{Y}$ is a
             collection of $C_0$-quasiconvex subsets and if $Q$ is another
             $C_0$-quasiconvex subset of $X$, then $Q$ is $\rho$-locally
             $10$-quasiconvex in $X^{el}_\YY$ (of course $\delta_0, C_0$ will be
             very small).

             To prove the claim, again, by contradiction, consider a sequence
             $X_N, \YY_N, Q_N$ of counter examples for $\delta_N=C_N= 1/N$ for
             $N=1,2,\dots$. There exist two points $x_N, y_N$
             in $Q_N$, at distance $\leq \rho$ from each other (for
             the electric metric), and a geodesic
             $ [ x_N, y_N ]$ in $X^{el}_\YY$ with a point $z_N$ on it at distance
             $>10$ from $Q$. We record  a point $z'_N$ in $Q$ at minimal distance
             ($\leq \rho $ in any case) from $z_N$.

             With a non principal ultrafilter $\omega$, we may take the
             asymptotic cone of the family of pointed spaces $(X_N^{el}, x_N)$. In
              $\lim_\omega ((X_N)^{el}_\YY,x_N)$, the sequences $(y_N)$,
             $([x_N, y_N])$ and $(z_N)$ have limits for which the distance
             inequalities persist, and we get that $\lim_\omega (Q_N,x_N)$ is not
             $\rho$-locally $10$-quasiconvex in
             $\lim_\omega ((X_N)^{el}_\YY,x_N)$. But as we noticed in
             Corollary \ref{coro;quasitree}
             $\lim_\omega ((X_N)^{el}_\YY,x_N)$ is a 2-quasi-tree of
             spaces of the form
             $( \lim_\omega (X_N, x'_N))^{el}_{\YY^\omega}$, which are the
             electrifications of  $\mathbb{R}$-trees $\lim_\omega
             (X_N, x'_N)$ over a
             family of convex subsets ({\it i.e.} subtrees).
             In this space, $\lim_\omega (Q_N, x_N)$
             is also a subforest of
             $\sqcup_{(u_N)\in \frakC} \lim_\omega (X_N, u_N)$.
             Also observe that if $(Q_N)$ is  visible in two adjacent
             classes,   then   $\lim_\omega (Q_N, x_N)$    is adjacent to their common cone point
             over sequences $  (Y_N^{(i)}), (Y_N^{(j)})  $.  
             Hence $\lim_\omega (Q_N, x_N)$ is $2$-quasiconvex in
             $ \lim_\omega ((X_N)^{el}_{ \YY^\omega}, x_N)$, and this
             contradicts the inequalities satisfied by
             $\lim_\omega \{x_N, y_N, z_N, z'_N\}$.  The claim is established for
             all $\rho$.

             Now there exists $\rho_0$ such that, in any $1$-hyperbolic space,
             any subset that is $\rho_0$-locally $10$-quasiconvex is
             $10^{14}$-globally quasiconvex (this classical fact, perhaps found
             elsewhere with other (better!) constants, follows also from the
             Gromov-Cartan-Hadamard theorem for instance). So, by choosing an
             appropriate $\rho$, we have proven that there is $\delta_0<1, C_0<1$
             and $C_1$, such that if $(X,d_X)$ is $\delta_0$-hyperbolic, if
             $\mathcal{Y}$ is a collection of $C_0$-quasiconvex subsets and if
             $Q$ is another $C_0$-quasiconvex subset of $X$, then $Q$ is
             $C_1$-quasiconvex in $X^{el}_\YY$.

             Coming back to the statement of the proposition, by rescaling our
             space, we have proven that if $(X,d_X)$ is a $\delta$-hyperbolic
             metric space with a collection $\mathcal{Y}$ of
             $C$-quasiconvex 
             sets, and if $Q$ is $C$-quasiconvex, then $Q$ is
             $\lambda C_1$-quasiconvex in $X_\YY^{el_\lambda}$ (as defined in the
             previous proof) for $\lambda = \max\{ \delta/\delta_0, C/C_0 \}$.
             Since $X_\YY^{el_\lambda}$ is quasi-isometric to $X_\YY^{el}$, by a
             $(\lambda,\lambda)$-quasi-isometry, it follows that $Q$ is
             $C'$-quasiconvex in $X_\YY^{el}$ for $C'$ depending only on
             $\delta, C$.

           \end{proof}

           Finally, we consider the proposed converse.

           \begin{prop*} \ref{prop;unfolding_qc}

             Let $(X,d)$ be hyperbolic, and let $\calY$ be a collection of
             uniformly quasiconvex subsets.  Let $H$ be a subset of $X$ that is
             coarsely path connected, and quasiconvex in the electrification
             $X_\calY^{el}$.

             Assume also that there exists $\epsilon\in (0,1)$, and $\Delta_0$
             such that for all $\Delta> \Delta_0$, wherever $H$
             $(\Delta,\epsilon)$-meets an item $Y$ in $\calY$, there is a path in
             $H^{+\epsilon \Delta}$ between the meeting points in $H$ that is
             uniformly a quasigeodesic in the metric $(X,d)$.

             Then $H$ is quasiconvex in $(X,d)$.

             The quasiconvexity constant of $H$ can be chosen to depend only on
             the constants involved for $(X,d), \calY, \Delta_0, \epsilon$, the
             coarse path connection constant, and the quasi-geodesic constant of
             the last assumption.

           \end{prop*}

           We use the same strategy again. The claim is now the
            lemma below. 
           To state it, we need to define the $m$-coarse path
           metric on an $m$-coarse path connected subspace of a metric
           space. A subset $Y\subset X$ of a metric space 
           is $m$-coarse path connected if for  any two
           points $x,y$ in it there is a sequence $x_0=x, x_1, \dots,
           x_r = y$ for some $r$ such that $x_i\in Y$ and $d_X(x_i,
           x_{i+1}) \leq m$ for all $i$. We call such a sequence an
           $m$-coarse path or a path with mesh $\leq m$.  The length of the coarse
           path $(x_0, \dots, x_r)$ is $\sum d_X(x_i,
           x_{i+1})$. The $m$-coarse path metric on $Y$ is the distance
           obtained by taking the infimum of  lengths of 
           coarse paths between its points. An $m$-coarse geodesic is a
           coarse path realizing the coarse path metric between two
           points.

           \begin{lemma}
             Fix $C_H^{el}$, $R>0$, $Q>1$, $\epsilon >0$, and
             $\Delta >10\epsilon$.
             Then there exists $\delta_0, C_0, m_0 >0$,  such that the
             following holds:\\  Assume that $(X,d)$ is geodesic,
             $\delta_0$-hyperbolic, with a collection $\calY$ of
             $C_0$-quasiconvex subsets. Further suppose that
              $H$ is an $m_0$-coarsely connected subset of $X$ which is 
             $C_H^{el}$-quasiconvex in the electrification $X_\calY^{el}$.
             Equip $H^{+\epsilon\Delta}$ with its $m_0$-coarse path metric $d_H$.

             Assume also that whenever $H$ $(\Delta, \epsilon)$-meets a set
             $Y \in \calY$,  there is a $(Q, C_0)$-quasigeodesic
             path,  
             which is an $(m_0/10)$-coarse path, in $H^{+\epsilon \Delta}$ joining the meeting points
             in $H$.

             Then for all $a,b \in H^{+\epsilon \Delta}$ at $d_H$-distance at
             most $R$ from each other, any $m_0$-coarse  $\delta_0$-quasi-geodesic
             of $H^{+\epsilon \Delta}$ (for its coarse path metric)   between $a,b$ is
             $(\Delta\times C_H^{el})$-close to a geodesic of $X$.
           \end{lemma}
           \begin{proof}
             Suppose that the claim is false: For all choice of $\delta, C, m$ there is
             a counterexample. 
             Set $\delta_N = C_N = 1/N$.

             For each $\epsilon$, there exists $N$ such that, in a
             $\frac{1}{N}$-hyperbolic space, for any two points $x,y$, and any
             $(Q, 1/N)$-quasigeodesic $p$ which is a $(1/N)$-coarse path between these two
             points, the $\epsilon$-neighborhood of $p$ contains
             the geodesics $[x,y]$. (This, for instance, is visible on an
             asymptotic cone).

             Thus, it is possible to choose a sequence $m_N>10/N$ decreasing to
             zero, such that pairs of $\frac{9\Delta }{10}$-long
             $(Q, 1/N)$-quasigeodesics with mesh $\leq 1/N$ in a
             $\frac{1}{N}$-hyperbolic spaces, with starting points at distance
             $\leq \Delta/10$ from each other, and ending points at distance
             $\leq \Delta/10$ from each other, necessarily lie at distance
             $(m_N/10)$ from one another.

             Let then $X_N, H_N, \calY_N$ be a counterexample to our claim for
             these values: for each $N$ there is $a_N, b_N$ in $H_N^{+\Delta}$,
             $R$-close to each other for $d_{H_N}$, and a point
             $c_N\in H_N^{+\Delta}$ in a coarse $\delta_N$-quasi-geodesic in
             $[a_N, b_N]_{d_{H_n}}$ at distance at least
             $(\Delta\times C_H^{el})$ from a geodesic $[a_N, b_N]$ in $X_N$.
             However $c_N$ is $ C_H^{el}$-close to a geodesic $[a_N, b_N]_{el}$
             in $X^{el}$.  Passing to an asymptotic cone, we find a map
             $p_\omega$ from an interval $[0,R]$  to a continuous path
              in $\lim_\omega (H_N, a_N)$    from $a^{\omega}$
             to $b^{\omega}$
(which can be equal to $a^{\omega}$) that passes through a point
             $c^{\omega}$ at distance $\geq (\Delta\times C_H^{el})$ from the arc
             $[a^{\omega}, b^{\omega}]$ in $\lim_\omega(X_N, a_N)$.

             However,
             it is at distance $\leq C_H^{el}$ in the electrification of
             $\lim_\omega(X_N, a_N)$ by $\calY^{\omega}$.

             It follows that on the path in  $\lim_\omega(X_N, a_N)$  from $[a^{\omega}, b^{\omega}]$ to
             $c^{\omega}$, there must exist a segment 
             of length $\geq \Delta$ 
             belonging to the same
             $Y^\omega \in \calY^{\omega}$ .  Let us say
             that $Y^\omega$ is the limit of a sequence $Y_N$. Note that the
             limit path $p_\omega$ crosses  this segment at least twice (once in
             either direction).

             Thus, for $N$ large enough, $H_N$ $(\Delta,\epsilon)$-meets $Y_N$,
             with two pairs of meeting points $(r_1, r_2), (s_1, s_2)$ in $H_N$,
             where $d(r_1,r_2) \geq 9\Delta/10$ (and $(s_1,s_2) \geq 9\Delta/10$)
             and $d(r_1,s_1) \leq 3\Delta/10$ and $ d(r_2,s_2) \leq 3\Delta/10$.
             By assumption, there is a $(Q, \frac{1}{N})$-quasi-geodesic path in
             $H^{+\epsilon \Delta}$ from $s_1$ to $s_2$ and another from $r_1$ to
             $r_2$, with mesh $<m_N/10$. They  have to fellow travel on a
             large subpath, and  pass at distance $\leq m_N/10$ from
             each other, by choice of $m_N$.
             One can therefore find a shortcut  that is still a path in
             $H^{+\Delta}$ of mesh $\leq m_N$,  a contradiction. 
           \end{proof}

           From the claim, we can prove the statement of the Proposition. Consider a situation as in the statement.  We may choose the coarse
           path connectivity constant of $H$ to be more than $10$ times the
           quasigeodesic constant of the last
           assumption there.  Take $\epsilon$, given by the assumption of the Proposition, and
           $\Delta >\max\{100 \epsilon, \Delta_0\}$. Let $Q$ be the
           quasi-geodesic constant given by the the assumption of the
           proposition on $(\Delta,\epsilon)$-meetings, and $C_H^{el}$ be as given
           by the assumption. Take $R$ larger (how large will be made clear  in the proof).

           Rescale the space $X$ so that the hyperbolicity constant, the
           quasiconvexity constant of items of $\calY$, and the constant of
           coarse path connection of $H$ are respectively smaller than 
           $\delta_0, C_0, m_0$ of the Lemma above.

           Note that the assumption of the proposition on
           $(\Delta,\epsilon)$-meetings is invariant under rescaling (except for
           the value of $\Delta_0$). Thus, this assumption still holds, with the
           same $\epsilon$, and for the specified $\Delta$ chosen above. The Lemma applies, and $H^{+\epsilon \Delta}$ is $R$-locally
           quasiconvex for the rescaled metric. By the local to global principle
           (in $\delta_0$ hyperbolic spaces), with a suitable preliminary (large enough) choice
           of $R$, $H^{+\epsilon \Delta}$ is then globally quasiconvex. After rescaling back to the
           original metric of $X$, $H^{+\lambda}$ is still quasiconvex for some
           $\lambda$ (depending on $\epsilon\Delta$, and the coefficient of
           rescaling); hence $H$ is quasiconvex.

           By construction, we also have the statement on the dependence of the
           quasiconvexity constant. $\Box$

           \subsubsection{Coarse hyperbolic embeddedness and Strong Relative
             Hyperbolicity}

           The following Proposition establishes the equivalence of Coarse
           hyperbolic embeddedness and Strong Relative Hyperbolicity in the
           context of this paper.
           \begin{prop}\label{prop;from_horo_to_he}
             Assume that $(X,d)$ is a metric space, and that $\YY$ is a
             collection of subspaces.

             If the horoballification $X_\YY^h$ of $X$ over $\YY$ is hyperbolic,
             then $\YY$ is coarsely hyperbolically embedded in the sense of
             spaces.

             If $X$ is hyperbolic and if $\YY$ is coarsely hyperbolically
             embedded in the sense of spaces, then $X_\YY^h$ is hyperbolic.
           \end{prop}

           We remark here parenthetically that the converse should be true
           without the assumption of hyperbolicity of $X$. However, this is not
           necessary for this paper.

           \begin{proof}

             Assume that the horoballification $X_\YY^h$ of $X$ over $\YY$ is
             $\delta-$hyperbolic.  The horoballs $Y^h$ (corresponding to $Y$) are
             thus $10\delta$-quasiconvex.  Therefore, by Proposition
             \ref{prop;criterion}, the electrified space obtained by electrifying
             (coning off) the horoballs $Y^h$ is hyperbolic.

             Since by Proposition \ref{double-elec}, this space is
             quasi-isometric to $(X^{el}_\YY, d^{el}_\YY)$, it follows that the
             later is hyperbolic. This proves the first condition of Definition
             \ref{he}.

             We want to prove the existence of a proper increasing function
             $\psi:\bbR_+\to\bbR_+$, such that the angular metric at each cone
             point $v_Y$ (for $Y\in \YY$) of $(X^{el}_\YY, d^{el}_\YY)$ is
             bounded below by $\psi \circ d|_Y$.  Define
             $$ \psi(r) = \inf_{Y\in \YY}\,\,\, \inf_{y_1, y_2 \in Y, d(y_1, y_2)
               \leq r } \hat{d}(y_1, y_2). $$

             Of course, the angular metric at $v_Y$ is bounded below by
             $\psi \circ d|_Y$.  The function $\psi$ is obviously increasing.  We
             need to show that it is proper, i.e. that it goes to $+\infty$.

             If $\psi$ is not proper, then there exists $\theta_0 > 0 $ such that
             for all $D$, there exist $Y\in \YY $ and $y, y' \in Y$ at
             $d-$distance greater than $D$ but $\hat{d} (y,y') \leq \theta_0$
             (where $\hat{d}$ is the angular metric on $Y$).  We choose
             $D >>\theta_0\delta$  (for instance $D= \exp(100 (\theta_0+1)(\delta+1))$).

             Consider a path in $(X^h_\YY)^{el}$ of length less than $\theta_0$
             from $y$ to $y'$ avoiding the cone point of $Y$.  Because
             $D>>\theta_0$, this path has to pass through other cone points. It
             can thus be chosen as a concatenation of $N+1$ geodesics whose
             vertices are $y,y'$ and some cone points $v_1, \dots, v_N$
             (corresponding to $Y_1, \dots, Y_N$ with $N<\theta_0$). Adjoining
             the (geodesic) path $[y,v_Y] \cup [v_Y, y']$ (where $v_Y$
             corresponds to the cone point for $Y$), we thus have a geodesic
             $(N+2)$-gon $\sigma$.
             Next replace each passage of $\sigma$ through a cone point ($v_i$ or
             $v_Y$) in $X^{el}_\YY$ by a geodesic ($\mu_i$ or $\mu_Y$
             respectively) in the corresponding horoball ($Y_i^h$ or $Y^h$
             respectively) in $X_\YY^h$ to obtain a geodesic $(2N+2)$-gon $P$ in
             $X_\YY^h$. The geodesic segments $\mu_i$ or $\mu_Y$ comprise $(n+1)$
             alternate sides of this geodesic $(2N+2)$-gon.

             Since $X_\YY^h$ is $\delta-$hyperbolic, it follows that the
             mid-point $m$ of $\mu_Y$ is at distance
             $\leq (2N+2)\delta$ from another edge of
             $P$.  Note that $m$ is in the horoball of $Y$, and because the
             distance in $Y$ between $y$ and $y'$ is larger than $D$, we have
             that $d^h(m, Y) $ is at least $\log(D)/2$.

             Since no other edge of $P$ enters the horoball $Y^h$, this forces
             $\log(D)$ (and hence $D$) to be bounded in terms of $\theta_0$
             and $\delta$: $D\leq \exp(4 (N +1) \delta)$. Since $N\leq \theta_0$, 
             this is a
             contradiction with the choice of $D$. We can conclude that $\psi$ is
             proper, and we have the first statement.

             Let us consider the second statement.  If $X$ is hyperbolic and if
             $\YY$ is coarsely hyperbolically embedded in the sense of spaces,
             then elements of $\YY$ are uniformly quasiconvex in $(X,d)$ by
             \ref{cor-retraction},
             and, by the property of the angular distance on any $Y\in \YY$, they
             are mutually cobounded.
             The statement then follows by Proposition \ref{prop;criterion}.
           \end{proof}

           \section{Algebraic Height and Intersection Properties}

           \subsection{Algebraic Height}
           We  
           recall here the general definition for height of finitely many subgroups. 

           \begin{defn} Let $G$ be a group and $\{ H_1, \dots, H_m\}$ be a finite collection of subgroups.
             Then the {\bf algebraic height} of this collection is 
             $n$ if
             $(n+1)$ is the smallest number with the property that
             for any $(n+1)$ distinct left cosets $g_1H_{\alpha_1},\dots,
             g_{n+1}H_{\alpha_{n+1}}$, the intersection $\bigcap_{1\leq i\leq n+1} g_i
             H_{\alpha_i} g_i^{-1}$ is  finite. \end{defn}

           We shall describe this briefly by saying that algebraic height is the largest $n$ for which the intersection of
           $n$ {\bf essentially distinct} conjugates of $ H_1, \dots, H_m$ is infinite. Here `essentially distinct' refers to
           the cosets of $ H_1, \dots, H_m$ and not to the conjugates themselves.

           \medskip

           For hyperbolic groups, one of the main Theorems of
           \cite{GMRS} is the following:

           \begin{theorem} \label{alght} \cite{GMRS} Let $G$ be a hyperbolic group and $H$ a quasiconvex subgroup. Then the algebraic height of $H$ is finite.
             Further, there exists $R_0$ such that if $H \cap gHg^{-1} $ is infinite, then $g$ has a double coset representative with length at most $R_0$.

             The same conclusions hold for finitely many quasiconvex subgroups $\{
             H_1, \dots, H_n\}$ of $G$. 
           \end{theorem}

           We quickly recall a proof of Theorem \ref{alght} for one subgroup $H$
           in order to generalize it to the context of mapping class groups and $Out(F_n)$.

           \begin{proof} 
             Let $G$ be hyperbolic, $X(=\Gamma_G)$ a Cayley graph of
             $G$ with respect to a finite generating set 
             (assumed to be $\delta-$hyperbolic),  and $H$ a
             $C_0-$quasiconvex subgroup of $G$. 
             Suppose that there exist $N$ essentially distinct
             conjugates $\{ H^{g_i}\}, i=1, \dots, N$, of $H$ that
             intersect  
             in an infinite subgroup. The $N$ left-cosets
             $g_iH$ are disjoint and
             share an accumulation point $p$ in the boundary of $G$ (in the limit set
             of $\cap_i H^{g_i}$). Since all $g_iH$ are $C_0-$quasiconvex, there exist
             $N$ disjoint quasi-geodesics $\sigma_1, \dots, \sigma_N$
             (with same constants $\lambda, \mu$ 
             depending only on $C_0, \delta$) converging to $p$ such
             that $\sigma_i$ is in $g_iH$. 
             Since $X$ is $\delta-$hyperbolic, there exists  $R
             (=R(\lambda, \mu, \delta) = R(C_0, \delta))$ and a point
             $p_0$
             sufficiently far along $\sigma_1$ such that all the
             quasi-geodesics $\sigma_1, \dots, \sigma_N$ pass through
             $B_R(p_0)$. Hence $N \leq \#(B_R(p_0))$
             giving us finiteness of height. 
             
             Further, any  such $\sigma_i$ furnishes a double coset representative
             $g_i^\prime$ of $g_i$ (say by taking a word that gives the shortest
             distance between
             $H$ and the coset $g_iH$) of length bounded in terms of
             $R$. This furnishes the second conclusion.
           \end{proof}

           \begin{remark} {\rm A word about generalizing the above
               argument to a family $\HH$ of finitely many subgroups
               is necessary.
               The place in the above argument where $\HH$ consists of
               a singleton is used essentially is in declaring that
               the 
               $N$ left-cosets
               $g_iH$ are disjoint. This might not be true in general
               (e.g. $H_1 < H_2$ for a family having two elements).
               However, by the pigeon-hole principle, choosing $N_1$ large enough, any 
               $N_1$ distinct conjugates $\{ H_i^{g_i}\}, i=1, \dots,
               N, H_i \in \HH$ must contain $N$ 
               essentially distinct conjugates $\{ H^{g_i}\}, i=1,
               \dots, N$ of some $H \in \HH$  and then the
               above argument for a single $H\in \HH$  goes through.}
           \end{remark}

           \begin{remark} {\rm A number of  other examples of finite
               algebraic height may be obtained 
               from certain special subgroups of
               Relatively Hyperbolic Groups, Mapping Class Groups and
               Out($F_n$).  These will be discussed after we introduce
               geometric height later in the paper.}
           \end{remark}

           \subsection{Geometric i-fold intersections}
           Given a finite family of subgroups of a group we  define
           collections of geometric $i-$fold intersections.

           \begin{defn}\label{def:ifoldinter}
             Let $G$ be endowed with a left invariant word metric
             $d$. Let $\HH$ be a finite   family  of subgroups of $G$.

             For $i\in \bbN, i\geq 2$,  define
             the {\bf geometric i-fold intersection}, or simply  the i-fold intersection of cosets of $\HH$,  $\HH_i$,  to be the set of subsets 
             $J$ of $G$ for which there exist $H_1, \dots, H_i \in
             \HH$ and  $g_1, \dots, g_i \in G$,  and $\Delta \in \mathbb{N}$ satisfying:

	     $$ J=  \left(  \bigcap_j \left(  g_jH_j   \right)^{+\Delta}   \right)  $$
	     	and   $ \bigcap_j \left(  g_jH_j   \right)^{+\Delta}  $ is not in the $20\delta$-neighborhood of   
                $ \bigcap \left(  g_jH_j   \right)^{+\Delta-2\delta}
                $, and the diameter of $J$ is at least
                $10\Delta$. 
	   \end{defn}

Geometric $i-$fold intersections are thus, by definition,
intersections of thickenings of cosets. The condition that the
diameter of the intersection is larger than $10$ times the thickening
is merely to avoid counting myriads of too small intersections. 

 The next proposition establishes that the collection of such intersections is again closed under intersection. 

	   \begin{prop}\label{geointnsstable}

             Consider $J\in \HH_j$, and $K\in \HH_k$ for $j<k$ and let
             $\Delta_J, \Delta_K$ be constants as in Definition
             \ref{def:ifoldinter} for defining $J$ and $K$ respectively. Write
             $J=    \left(  \bigcap_i \left(  g_iH_i
               \right)^{+\Delta_J}   \right) $ and $K =  \left(
               \bigcap_i \left(  g'_iH'_i   \right)^{+\Delta_K}
             \right) $, and 
             let  $\Delta_0>\max (\Delta_J, \Delta_K)$.

             Assume that $J$ and $K$   
             $(\Delta,\epsilon)$-meet, 
                for some $\Delta>20\Delta_0$,  and for $\epsilon <1/50$.

                Then either $K\subset J$, or for   any pair of    
                $(\Delta, \epsilon)$-meeting points of $J$ and $K$,   
                there is $L\in \HH_{j+1}$ contained in $J$, 
                that contains it.        
	   \end{prop}

\begin{proof}

Let $x,y $ be    $(\Delta, \epsilon)$-meeting points of $J$ and $K$.   If   $K\not\subset J$,   
  we can assume that $x,y$ are in $\left(  g'_1H'_1   \right)^{+\Delta_K+\epsilon\Delta}  $ for some $g'_1H'_1  $ not contained in          the collection      $ \{  g_iH_i   \}$.

Notice that $x,y$ are in $  \bigcap_i \left(  g_iH_i   \right)^{+ \Delta_J +\epsilon\Delta} \cap   \left(  g'_1H'_1   \right)^{+\Delta_K +\epsilon\Delta} $, 
hence in $  \bigcap_i \left(  g_iH_i   \right)^{+\Delta'  } \cap   \left(  g'_1H'_1   \right)^{+ \Delta'} $ for $\Delta'$ the greater of $\Delta_K +\epsilon \Delta$ and $\Delta_J + \epsilon \Delta$.

We argue
by contradiction. Suppose that  $x,y$ are contained in the $20\delta$-thickening of a $2\delta$-lesser intersection. It follows that there are   $x',y'$ such that $d(x,x')\leq 20\delta$, $d(y,y') <20\delta$ and still $d(x', J) \leq \epsilon \Delta -2\delta$, and $d(y', J) \leq \epsilon \Delta -2\delta$.

 But by definition of $(\Delta,\epsilon)$-meeting, this is a
 contradiction.   

Finally, 
 the diameter of the intersection of $\Delta'$-thickenings of
 our cosets, is larger than  $  \Delta  $.  
 Since the thickening constant is
 $\Delta' \leq \Delta_0+ \epsilon\Delta$, the ratio of 
 the thickening constant by the diameter is at most  $(\Delta_0+
 \epsilon\Delta)/\Delta$ which is less than $ 1/10$, hence the
 result.

\end{proof}

           Let $(G,d)$ be a group with a word metric and $H <G$ a
           subgroup. The restriction of $d$ on $H$  will be called the {\bf induced metric} on $H$ from $G$.

           \begin{prop} \label{prop;projections} Let $(G,d)$ be a
             group with a $\delta$-hyperbolic word metric (not
             necessarily locally finite). 

             Assume that $A_1,\dots, A_n$ are
             $C$-quasiconvex subsets of $G$. 
             Then for all $\Delta >C+20\delta$, the intersection
             $\bigcap A_i ^{+\Delta}$ is $(4\delta)$-quasiconvex in
             $(G,d)$. 

             Moreover, if   $A$ and $B$ are
             $C$-quasiconvex subsets of $G$, and    if $\Pi_B(A)$
             denotes the set of nearest points projections
             of $A$ on $B$, then, either $\Pi_B(A) \subset A^{+3C
               +10\delta}$ or ${\rm Diam} \Pi_B(A)
             \leq  4C+20\delta$. 
           \end{prop}

           \begin{proof}
             Consider $x,y \in \bigcap A_i^{+\Delta}$ and take $a_i, b_i$
             some nearest point projection on $A_i$. On a
             geodesic $[x, y]$ take $p$ at distance greater than $4\delta$
             from $x$ and $y$. Hyperbolicity applied to the  quadrilateral
             $(x, a_i, b_i, y)$  tells us that $x$ is $4\delta$-close to $[x'_i,
             a_i]\cup [a_i, b_i] \cup [b_i, y'_i]$, where $x'_i$ and $y'_i$ are the
             points of, respectively $[x, a_i]$ and $[y, b_i]$, at distance
             $4\delta$ from, respectively, $x$ and $y$. 

             Let us call $[x'_i,
             a_i],   [b_i, y'_i] $ 
             the approaching segment, and   $[a_i, b_i]$ the traveling
             segments. Hence for each $i$, $p$ is closed to either an approaching
             segment, or the traveling segment, with subscript $i$.  

             If $p$ is close to  
             an approaching segment of index $i$, then it is in  $A_i^{+\Delta}$.

             If $x$ is close to  the traveling segment of index $i$, then  it is at
             distance at most $C+10\delta$  from  $A_i$, hence in $A_i^{+\Delta}$
             because  $\Delta>C+10\delta$.

             We thus obtain that $[x,y]$ remains at distance $4\delta$
             from $\bigcap A_i^{+\Delta}$.

             To prove the second statement, take $a_0, b_0$ in
             $A$ and $B$ respectively realizing the distance (up to $\delta$ if
             necessary). Let $b \in \Pi_B(A)$, and assume that it is the projection of
             $a$. In the quadrilateral $a,a_0,b,b_0$, the geodesic $[a,a_0]$ stays
             in $A^{+C}$ and $[b,b_0]$ is in $B^{+C}$. Since $b$ is a
             projection, $[b,a]$ fellow-travels $[b,b_0]$ for less than $2C+10\delta$, and
             similarly for $[b_0,a_0]$ with $[b_0,b]$. By hyperbolicity $[b,b_0]$
             thus stays $10\delta$ close to $[a,a_0]$ except for the part
             $(2C+10\delta)$-close to either $b$ or $b_0$. 
             It follows that either $b \in A^{+(3C  +10 \delta)}$ or
             $b$ is at distance $\leq 4C+20\delta$ from $b_0$.  
             Thus  $\Pi_B(A) \subset A^{+3C
               +10\delta}$ 
             or ${\rm Diam} \Pi_B(A)
             \leq  4C+20\delta$.

           \end{proof}

           \subsection{Algebraic i-fold intersections}
           We provide now a more algebraic (group theoretic) treatment
           of the preceding discussion. 
           This is in keeping with the more well-known setup of intersections of subgroups
           and their conjugates cf. \cite{GMRS}.
           Given a finite family of subgroups of a group we first
           define collections of $i-$fold conjugates 
           or algebraic i-fold intersections.
           \begin{defn}\label{essdist}
             Let $G$ be endowed with a left invariant word metric
             $d$. Let $\HH$ be a family  of subgroup of $G$. 
              For $i\in \bbN, i\geq 2$,  define $\HH_i$ to be the set of subgroups 
             $J$ of $G$  for which there exists $H_1, \dots, H_i \in
             \HH$ and  $g_1, \dots, g_i \in G$ satisfying: 
             \begin{itemize}
             \item the cosets $g_j H_j$  are pairwise  distinct (and
               hence as in \cite{GMRS} we use the terminology that the
               conjugates 
               $\{g_j H_jg_j^{-1}, j=1,\dots,i\}$
               are {\bf essentially distinct})
             \item $J$ is the intersection $\bigcap_j g_j H_j g_j^{-1}$. 
             \item $J$ is unbounded  in $(G,d)$.
             \end{itemize}
             We shall call $\HH_i$ the family of {\bf  algebraic i-fold intersections}
             or simply, {\bf $i-$fold conjugates}.
           \end{defn}

           The second point in the following definition is motivated
           by the behavior of nearest point projections  of  
           cosets of quasiconvex subgroups of hyperbolic groups on
           each other. Let $(G,d)$ be hyperbolic and $H_1, H_2$ 
           be quasiconvex. Let $aH_1, bH_2$ be cosets and
           $c=a^{-1}b$. Then the nearest point projection of $bH_2$
           onto $aH_1$ is the (left) $a-$translate of the nearest
           point projection of $cH_2$ onto $H_1$.  
           Let $\Pi_B (A)$ denote the (nearest-point) projection of
           $A$ onto $B$. Then $\Pi_{H_1} (cH_2)$ lies in a bounded 
           neighborhood (say $D_0-$neighborhood) of $H_2^c\cap H_1$
           and so $\Pi_{aH_1} (bH_2)$  lies in a 
           $D_0-$neighborhood of $bH_2 b^{-1}a \cap aH_1$. The latter does lie
           in a bounded neighborhood of $(H_2)^b \cap (H_1)^a$, but
           this bound depends on $a, b$ and is not uniform. Hence the
           somewhat convoluted way of stating the second property
           below. The language of nearest-point projections below is
           in
           the spirit of \cite{mahan-ibdd, mahan-split} while the
           notion of geometric i-fold intersections discussed earlier 
           is in the spirit of \cite{DGO}.

           \begin{defn} Let $G$ be a group and $d$ a word
             metric   on $G$. 

             A finite family $\calH = \{ H_1, \dots , H_m \}$  of
             subgroups of $G$, each equipped with a word-metric $d_i$
             is said to 
             have the {\bf uniform qi-intersection property}
             if there exist $C_1, \dots, C_n , \dots$ such that 
             \begin{enumerate}
             \item For all $n$, and all  $H \in \HH_n$, 
               $H$  has a conjugate $H'$ such that if $d'$ is any {\bf
                 induced metric} on $H'$ from some $H_i \in \HH$, then
               $(H',d')$  is  $(C_1,C_1)-$qi-embedded in $(G,d)$.

             \item   For all $n$, let $(\HH_n)_0$ be a choice of conjugacy representatives
               of elements of $\HH_n$ that are $C_1$-quasiconvex in $(G,d)$.
               Let $\CC\HH_n$ denote the collection of left cosets of
               elements of $(\HH_n)_0$.

               For all $A, B \in \CC\HH_{n}$ with $A=aA_0, B=bB_0$,
               and $A_0, B_0 \in (\HH_n)_0$, 
               $\Pi_B (A)$ either has diameter bounded by $C_n$ for the metric
               $d$,  or   
               $\Pi_B (A)$ lies in a (left) $a-$translate translate of a 
               $C_n-$neighborhood  of $A_0\cap B_0^c$, where $c=a^{-1}b$.  
             \end{enumerate}  
             \label{qiintn} \end{defn} 

           In keeping with the spirit of the previous subsection, we
           provide a geometric version of the above definition below.
           \begin{defn} Let $G$ be a group and $d$ a word
             metric   on $G$. 

             A finite family $\calH = \{ H_1, \dots , H_m \}$  of
             subgroups of $G$, each equipped with a word-metric $d_i$
             is said to 
             have the {\bf uniform geometric qi-intersection property}
             if there exist $C_1, \dots, C_n , \dots$  such that 
             \begin{enumerate}
             \item For all $n$, and all  $H \in \HH_n$, 
               $(H,d)$  is  $C_n$-coarsely path connected, and
               $(C_1,C_1)-$qi-embedded in $(G,d)$ (for its coarse path metric).

             \item For all $A, B \in \HH_{n}$   
               either ${\rm diam}_{G,d} ( \Pi_B (A) ) \leq C_n$,  or   
               $\Pi_B (A) \subset A^{+C_n} $ for $d$.

             \end{enumerate}  
             \label{geoqiintn} \end{defn}

           \begin{remark}{\rm
               The second condition of Definition \ref{geoqiintn}
               follows from the first condition
               if $d$ is hyperbolic by Proposition \ref{prop;projections}.
               Further, the first condition holds for such $(G,d)$ so long as $\Delta$ is taken
               of the order of the quasiconvexity constants (again by
               Proposition \ref{prop;projections}). 

               Note further that if $G$ is hyperbolic (with respect to
               a not necessarily locally finite word
               metric) and $H$ is $C$-quasiconvex, then  
               by Proposition \ref{geointnsstable}
               the collection of geometric $n$-fold intersections
               $\HH_n$ is      mutually cobounded for the metric of
               $(G,d)^{el}_{\HH_{n+1}}$   (as in Definition
               \ref{cobdd}).  }
           \label{rem:cobdd}
           \end{remark}

           \subsection{Existing results on algebraic intersection properties} 
           We start with the following result due to Short. 
           \begin{theorem} \cite[Proposition 3]{short} Let $G$ be a
             group generated by the finite set $S$. 
             Suppose $G$   acts properly on a uniformly proper geodesic
             metric space $(X,d)$, with a base point $x_0$.
             Given $C_0$, there exists $C_1$ such that
             if $H_1, H_2$ are  subgroups of $G$  for
             which the orbits $H_ix_0$ are $C_0$-quasiconvex in $ (X,d)$ (for
             $i=1,2$) then the orbit  $ (H_1 \cap H_2)x_0$ is $C_1-$quasiconvex in $(X,d)$. 
             \label{short-intn} 
           \end{theorem}

           We remark here that in the original statement of
           \cite[Proposition 3]{short}, $X$ is itself the Cayley graph
           of $G$ with respect to $S$, but the proof there goes
           through
           without change to the general context of Proposition \ref{short-intn}.

           In particular, for $G$ (Gromov) hyperbolic, or
           $G=Mod(S)$ acting on Teichmuller space $Teich(S)$ (equipped with the Teichmuller metric)
           and
           $Out(F_n)$ acting on Outer space $cv_N$ (with the symmetrized Lipschitz metric),
           the notions of (respectively) quasiconvex subgroups or
           convex cocompact subgroups of $Mod(S)$ or $Out(F_n)$ (see Sections \ref{mcg} and
           \ref{outfn} below for the Definitions) are
           independent of the finite generating sets chosen. Hence
           we have the following.

           \begin{theorem} 
             Let $G$ be either $Mod(S)$ or $Out(F_n)$
             equipped with some   finite generating set. 
             Given $C_0$, there exists $C_1$ such that
             if  $H_1, H_2$ are $C_0-$convex cocompact subgroups of $G$, 
             then $H_1 \cap H_2$ is $C_1-$convex cocompact in
             $G$. \label{short-intn-coco} 
           \end{theorem}

           The corresponding statement for relatively hyperbolic
           groups and relatively quasiconvex groups is due to Hruska.
           For completeness we recall it.

           \begin{defn}\label{relqc} \cite{osin-relhyp, Hru} 
             Let $G$
             be finitely generated hyperbolic relative to a finite
             collection  $\PP$ of  parabolic subgroups.
             A subgroup $H \le G$ is {\bf relatively quasiconvex} if the following holds.\\
             Let $S$ be some (any) finite relative generating set for $(G,\PP)$,
             and let $P$ be the union of all $P_i \in \PP$.
             Let $\overline\Gamma$ denote the Cayley graph of $G$ with
             respect to the generating set $S\cup P$
             and $d$ the word metric on  $G$.
             Then there is a constant $C_0=C_0({S},d)$ such that
             for each geodesic $\gamma \subset \overline{\Gamma}$
             joining two points of $H$,
             every vertex of $\gamma$ lies within $C_0$ of $H$
             (measured with respect to $d$).
           \end{defn}

           \begin{theorem} \cite{Hru} 
             Let $G$ be finitely generated
             hyperbolic relative to $\PP$.
             Given $C_0,$ there exists $C_1$ such that  if $H_1, H_2$
             are $C_0-$relatively quasiconvex subgroups of $G$, 
             then $H_1 \cap H_2$ is $C_1-$relatively quasiconvex in
             $G$. \label{hruska-intn} 
           \end{theorem}

           \section{Geometric height and graded geometric relative hyperbolicity}

           We are now in a position to define the geometric analog of
           height. There are two closely related notions
           possible, one corresponding to the geometric notion of 
           $i-$fold
           intersections and one corresponding to the algebraic notion of 
           $i-$fold conjugates. The former is relevant when one
           deals with subsets and the latter when one
           deals with subgroups.
           \begin{defn} \label{gh}
             Let $G$ be a group, with a left invariant word metric $d
             (=d_G)$ with respect to some (not necessarily finite) generating set.
             Let $\HH$ be a family of subgroups of $G$. 

             The {\bf  geometric height}, 
             of $\HH$ in $(G,d)$ (for $d$) is the minimal number
             $i\geq 0$ so that the collection   $\HH_{i+1}$ of $(i+1)-$fold
             intersections consists of uniformly bounded sets. 

             If $H$ is a single subgroup, its geometric height is that of the
             family $\{H\}$.
           \end{defn}

           \begin{remark}\label{comp1}
             {\bf Comparing notions of height:} 
             \begin{itemize}
             \item Geometric 
               height is related to   algebraic height, but is more flexible, since
               in the former, we allow the group $G$ to have an infinite generating
               set. 
               We are then free to apply the operations of electrification, horoballification
               in the context of non-proper graphs.
             \item In the case of a locally finite
               word metric,  algebraic height is less than or equal to  geometric 
               height. Equality holds if all bounded intersections are
               uniformly bounded.

             \item  For a locally finite
               word metric,    finiteness of algebraic height
               implies that $i-$fold conjugates are finite (and hence
               bounded
               in any metric) for all sufficiently large $i$. Hence 
               finiteness of 
               geometric height 
               follows from finiteness of algebraic height and of a
               uniform bound on the diameter of the finite
               intersections. 
             \item When the metric on a Cayley graph is not locally finite,
             we do not know of any general statement that allows us to go  directly from finiteness 
          of diameter of an  intersection of thickenings of cosets (geometric condition)
          to  finiteness of diameter of intersections
          of conjugates (algebraic condition). Some of the technical complications below are due to this difficulty in 
          going from geometric intersections to algebraic intersections.
             \end{itemize}
           \end{remark}

           We generalize Definition \ref{grh} of graded relative hyperbolicity
           to the context of geometric height as follows.

           \begin{defn} \label{ggrh}
             Let $G$ be a  group, $d$ the word metric with respect to
             some (not necessarily finite) generating set 
             and $\calH$ a finite collection of subgroups.             
               Let $\HH_i$ be the collection of  all 
               $i-$fold conjugates of $\HH$. Let $(\HH_i)_0$ be a choice of conjugacy
               representatives, and $\CC\HH_i$ the set of left cosets   of elements
               of $(\HH_i)_0$
               Let $d_i$ be the metric on $(G,d)$ obtained by electrifying
               the elements of $\CC\HH_i$.
               Let $\CC\calH_\bbN$ be the graded family
               $(\CC\calH_i)_{i\in \mathbb{N}}$.  

               We say that $G$ is 
               {\bf graded geometric relatively hyperbolic} with
               respect to $\CC\calH_\bbN$  
               if

               \begin{enumerate}
               \item $\calH$ has  
   geometric height $n$ for some $n \in \natls$, and for each $i$
   there are finitely many orbits of $i$-fold intersections.
               \item For all $i\leq n+1$,    $\CC\HH_{i-1}$ is
                 coarsely hyperbolically embedded in
                 $(G,d_i)$.  
               \item  There is $D_i$
                 such that all items of $\CC\HH_i$ are $D_i$-coarsely
                 path connected in $(G,d)$.
               \end{enumerate}
           \end{defn}

           \begin{remark} \label{comp2}
             {\bf Comparing geometric and algebraic graded relative hyperbolicity:}\\
             Note that            
              the second condition of
             Definition \ref{ggrh} is equivalent, by Proposition
             \ref{prop;from_horo_to_he}, to saying that
             $(G,d_i)$ is strongly hyperbolic relative to the collection
             $\HH_{i-1}$. This is exactly the
             third (more algebraic) condition in Definition \ref{grh}.
             Also,            
             the third condition 
             of Definition \ref{ggrh} is the analog of (and follows from)
             the second (more algebraic) condition in Definition \ref{grh}.

             Thus finite 
              geometric height along with (algebraic)
             graded relative hyperbolicity implies
              graded geometric
             relative hyperbolicity.
         
           \end{remark}

           The rest of this section furnishes  examples of  finite
           height in both its geometric and algebraic incarnations.

           \subsection{Hyperbolic groups}
           \begin{prop}\label{prop;hyp_qc_have_fgh}
             Let $(G,d)$ be a hyperbolic group with a locally finite word metric, and 
             let $H$ be a quasiconvex subgroup of $G$. Then  $H$ has finite
         geometric height.

             More precisely, if $C$ is the quasi-convexity constant of $H$ in
             $(G,d)$, and if $\delta$ be
             the hyperbolicity constant in $(G,d)$, and if $N$ is  the cardinality of
             a ball of $(G,d)$  of radius $2C+10\delta$,  and if   $g_0H, \dots, g_k
             H$ are distinct cosets of $H$ for which there exists $\Delta$ such
             that the total intersection  $\bigcap_{i=0}^k (g_i H)^{+\Delta}$ has diameter more
             than $10\Delta$, and more that $100\delta$, then there exists $x \in
             G$ such
             that each $g_i H$ intersects the ball of radius  $N$ around
             $x$. 

           \end{prop}

           First note that the second statement implies the first in the
          (by the third
           point of Remark \ref{comp1}).   We will directly
           prove the second. 
           The proof is similar to the finiteness of the algebraic height. Also
           note that the second statement can be rephrased in terms of double
           cosets representatives of the $g_i$: under the assumption on the total
           intersection, and if $g_0=1$,  there are double coset
           representatives of the $g_i$ of
           length at most $2(2C+10\delta)$.   

           \begin{proof}
             Assume that there exists $\Delta>0$, and elements $1=g_0, g_1, \dots, g_k$
             for which the cosets  $g_i H$  are
             distinct, and    $\bigcap_{i=0}^k (g_i H)^{+\Delta}$ has diameter larger
             than $10\Delta$ and than $100\delta$. 

             First we treat the case $\Delta >5\delta$. 

             Pick $y_1, y_2 \in \bigcap_{i=0}^k (g_i H)^{+\Delta}$ at distance
             $10\Delta$ from each
             other, and pick $x\in [y_1, y_2]$ at distance larger  
             than  $\Delta +10\delta$ from both $y_i$. For each $i$ an application
             of hyperbolicity and quasi-convexity tells us that $x$ is at distance at
             most $2C +10\delta$ from each of  $g_i H$. The ball of radius  $2C
             +10\delta$ around $x$ thus meets each coset $g_iH$.

             If $\Delta \leq 5\delta$, we pick $y_1, y_2 \in
             \bigcap_{i=0}^k (g_i H)^{+\Delta}$ at distance
             $100\delta$ from each
             other, and take $x$ at distance greater than $10\delta$ from both
             ends. The end of the proof is the same.

           \end{proof}

           \subsection{Relatively hyperbolic groups}

           If $G$ is hyperbolic relative to a collection of subgroups
           $\calP$, then Hruska and Wise defined in \cite{hruska-wise}
           the relative height of
           a subgroup $H$ of $G$ as $n$ if $(n+1)$ is the smallest number with
           the property that for any $(n+1)$ elements $g_0, \dots, g_n$ such that
           the  $g_iH$ are $(n+1)$- distinct cosets,   the intersections of
           conjugates $\bigcap_i g_i H g_i^{-1}$ is finite or parabolic.

           The notion of relative algebraic height is actually  the geometric 
           height for the relative distance, which is given by a word
           metric over a generating set that is the
           union of a finite set and  a set of conjugacy representatives of the
           elements of $\calP$.  
           Indeed, in a relatively hyperbolic group, the
           subgroups that are bounded in the relative metric are precisely those
           that are finite or
           parabolic. We give a quick argument.
           It follows from the Definition of relative quasiconvexity
            that a subgroup having finite diameter in the electric metric on $G$
           (rel. $\PP$) is relatively quasiconvex. It is also true \cite{DGO} that the normalizer of any $P \in \PP$
           is itself  and that the subgroup generated by any $P$ and any infinite order
           element $g\in G \setminus P$ contains the free product of conjugates of $P$ by $g^{kn}, k \in \Z$.
           Since any proper supergroup of $P$ necessarily contains such a $g$, it follows that no proper supergroup
           of $P$ can be of finite diameter in the electric metric on $G$
           (rel. $\PP$). It follows that bounded subgroups are precisely the finite subgroups or
           those contained inside parabolic subgroups.

           The notion of relative height can 
           actually 
           be extended to define the height of a collection of
           subgroups $H_1, \dots, H_k$, as in the case for the algebraic height.  

           Hruska and Wise proved that relatively quasiconvex subgroups have finite relative
           height. More precisely:

           \begin{theorem}\cite[Theorem 1.4, Corollaries 8.5-8.7]{hruska-wise} \label{relht}
             Let $(G,\PP)$
             be relatively hyperbolic, let $S$ be a finite relative
             generating set for $G$ and $\Gamma$ be the Cayley graph
             of $G$ with respect to $S$.
             Then for $\sigma \geq 0$, there exists $C\geq 0$
             such that the following holds. \\ Let
             $H_1, \dots , H_n$
             be a finite collection of $\sigma-$relatively quasiconvex
             subgroups of $(G,\PP)$.
             Suppose that there exist distinct cosets $\{
             g_mH_{\alpha_m} \}$ with $\alpha_m \in \{1, \dots ,
             n\}$, $m = 1, \dots, n$,
             such that
             $\cap_m g_mH_{\alpha_m}g_m^{-1}$ is not contained in a
             parabolic $P\in \PP$. Then there exists a vertex
             $z \in G$
             such that the ball of radius
             $C$
             in $\Gamma$
             intersects every coset $g_mH_{\alpha_m}$. 

             Further, for any $i \in \{1, \dots , n\}$, there are  only finitely many double cosets of the form
             $H_i g_i H_{\alpha_i}$
             such that
             $H_i \cap \bigcap_i  g_i H_{\alpha_i}g_i^{-1}$
             is not contained in a parabolic $P\in \PP$.

             Let $G$ be a relatively hyperbolic group, and let $H$ be a relatively quasiconvex subgroup.
             Then $H$ has finite relative algebraic height.\end{theorem}

           This allows us to  give an example of geometric height in our setting.

           \begin{prop}\label{prop;rh_or_gh}

             Let  $(G, \calP)$ be a relatively hyperbolic group, and
             $(G,d)$  a relative word metric $d$ ({\it i.e.} a word
             metric over a generating set that is the
             union of a finite set and of a set of conjugacy representatives of the
             elements of $\calP$, and hence, in general, not a finite generating
             set). Let $H$ be a relatively quasiconvex subgroup. Then, $H$ has
             finite  geometric height for $d$.

           \end{prop}

           This just  a rephrasing of Hruska and Wise's
           result Theorem \ref{relht}. The proof is similar to
           that in the hyperbolic groups case, using for
           instance cones instead of balls.

           \subsection{Mapping Class Groups}\label{mcg}
           Another source of examples arise from convex-cocompact subgroups of
           Mapping Class Groups, and of $Out(F_n)$ for a free group $F_n$. We establish
           finiteness of both algebraic and geometric height 
           of convex cocompact subgroups of Mapping Class Groups in
           this subsubsection.  In the following $S$ will be a closed
           oriented surface of genus greater than $2$, and 
		$Teich(S)$ and $CC(S)$ will denote respectively the Teichmuller space and Curve Complex of $S$.

           \begin{defn} \cite{farb-mosher}
             A finitely generated subgroup $H$ of the mapping class group
             $Mod(S)$ for a surface $S$ (with or without punctures) is
             $\sigma-$convex cocompact if for some (any) $x\in  Teich(S)$, the
             Teichmuller space of $S$, the orbit $Hx \subset Teich(S)$ is
             $\sigma-$quasiconvex with respect to the Teichmuller metric. 
           \end{defn}

           Kent-Leininger \cite{kl} and Hamenstadt \cite{ham-cc}  prove the following:

           \begin{theorem}\label{theo;klh}
             A  finitely generated subgroup $H$ of the mapping class group
             $Mod(S)$ is convex cocompact  if and only if for some (any) $x\in
             CC(S)$, the curve complex of $S$, the orbit $Hx \subset CC(S)$ is
             qi-embedded in $CC(S)$.\label{qi-coco}
           \end{theorem}

           One important ingredient in Kent-Leininger's proof of Theorem
           \ref{qi-coco} is a lifting of the limit set of $H$ in $\partial CC(S)$
           (the boundary of the curve complex) to $\partial Teich(S)$ (the boundary
           of Teichmuller space). What is important here is that $Teich(S)$ is a
           proper metric space unlike $CC(S)$. Further, they show using  a
           Theorem of Masur \cite{masur}, that any two Teichmuller geodesics
           converging to a point on the limit set $\Lambda_H$ (in $\partial
           Teich(S)$) of a convex cocompact subgroup $H$ are asymptotic. An
           alternate proof is given by Hamenstadt in \cite{ham-gd}. With these
           ingredients in place, the proof of Theorem \ref{alght-mcg0} below is
           an exact replica of the proof of Theorem \ref{alght} above:

           \begin{theorem} (Height from the Teichmuller metric) \label{alght-mcg0}  
             Let $G$ be the mapping class group of a surface $S$, and
             $Teich(S)$ the corresponding Teichmuller space with the
             Teichmuller metric, and with a
             base point $z_0$. Then for
             $\sigma \geq 0$, there exists $C\geq 0$, and $D\geq 0$ such that the
             following holds.

             Let $H_1, \dots , H_n$ be a finite collection of $\sigma-$convex
             cocompact subgroups of $G$. Suppose that there exist distinct
             cosets $\{ g_mH_{\alpha_m} \}$ with $\alpha_m \in \{1, \dots ,
             n\}$, $m = 1, \dots, n$, such that, for some $\Delta$,  
             $\cap_m (g_mH_{\alpha_m})^{+\Delta}$ is larger than $\max\{10\Delta,
             D\}$.  Then there exists a point $z \in Teich(S)$ such that the
             ball of radius $C$ in $Teich(S)$ intersects every image
             of $z_0$ by a coset    $g_mH_{\alpha_m} z_0$. 
             
             Further, for any $i \in \{1, \dots , n\}$, there are  only
             finitely many double cosets of the form $H_i g_i H_{\alpha_i}$
             such that $H_i \cap \bigcap_i  g_i H_{\alpha_i}g_i^{-1}$
             is infinite.

             The collection $\{ H_1, \dots , H_n\}$ has finite algebraic height.
           \end{theorem}

           A more geometric strengthening of Theorem \ref{alght-mcg0} can be
           obtained as follows using recent work of Durham and Taylor
           \cite{DuTa}, who have given an intrinsic quasi-convexity
           interpretation of convex cocompactness, by proving that convex
           cocompact subgroups of Mapping Class Groups are stable: in
           a word metric,  they are
           undistorted, and  quasi geodesics with end points in the subgroup remain close
           to each other \cite{DuTa}.

           \begin{theorem} \label{alght-mcg}  (Height from a word
             metric) 
             Let $G$ be the mapping class group of a surface $S$ and $d$ the word
             metric with respect to a finite generating set.
             Then for $\sigma \geq 0$, and any subgroup $H$ that is
             $\sigma$-convex cocompact, the group $H$ has finite  geometric height in $(G,d)$.

             Moreover, any $\sigma$-convex cocompact subgroup $H$ has finite
             geometric height in  $(G,d_1)$, where 
             $d_1$ is the word metric with respect to any (not
             necessarily finite) generating set.
             \end{theorem}
           
           \begin{proof}

             Assume that the theorem is false:  there exists $\sigma$ such
             that for all $k$, and all $D$, there exists a
             $\sigma$-convex
             cocompact subgroup   $H$, with a collection of distinct cosets $\{ g_mH,
             m=0, \dots, k
             \}$ (with $g_0=1$),  satisfying the property that $\cap_m (g_mH)^{+\Delta}$  
has diameter larger than
             $\max\{10\Delta, D\}$.   

             Let $a,b$ be two points in $\cap_m (g_mH)^{+\Delta}$  such that
             $d(a,b) \geq \max\{10\Delta, D\}$. For each $i$, let $a_i, b_i$ in
             $g_i H$ be at distance at most $\Delta$ from $a$ and $b$
             respectively. Consider $\gamma_i$  geodesics in $H$ from  $g_i^{-1}
             a_i$ to $g_i^{-1} b_i$.    Consider also $a'_i$ and $b'_i-$ nearest
             point projections of $a_0$ and $b_0$ on  $g_i \gamma_i$. Finally,
             denote by $g_i \gamma'_i$ the subpath of $g_i \gamma_i$
             between   $a'_i$ and $b'_i$

             By  \cite[Prop. 5.7]{DuTa}, $H$ is quasiconvex in $G$ (for a fixed chosen
             word metric), and
             for each $i$, $g_i \gamma_i$ is a
             $f(\sigma)$-quasi-geodesic (for some function $f$).

             We thus obtain from $a_0$ to $b_0$ a family of  $k+1$ paths, namely
             $\gamma_0$ and (one for all $i$), the concatenation $ \eta_i = [a_0, a'_i] \cdot
             g_i \gamma'_i \cdot [b'_i,b_0]$.

             For $D$ large enough, the paths $\eta_i$ are
             $2f(\sigma)$-quasigeodesics  in $G$.

             Stability of $H$ (\cite[Thm. 1.1]{DuTa}) implies 
             that, there exists $R(\sigma)$ such that in $G$, the paths remain
             at mutual  Hausdorff distance at most 
             $R(\sigma)$.    This is thus also true in the
             Teichmuller space by the orbit map.  Hence it follows that all the
             subpaths    $g_i \gamma'_i$ are at distance at most  $2 R(\sigma)$
             from each other, but are disjoint, and all lie in a thick part of the
             Teichmuller space, where the action is uniformly proper. This leads to
             a contradiction.

             Since the diameter of intersections can only go down if the
             generating set is increased, the last statement follows.
           \end{proof}

           \subsection{Out$(F_n)$}  \label{outfn}

           Following Dowdall-Taylor \cite{dt1}, we say that 
           a finitely generated  subgroup $H$ of  $Out(F_n)$
           is
           $\sigma-${\bf convex cocompact}
           if 
           \begin{enumerate}
           \item all non-trivial elements of $H$ are atoroidal and fully irreducible.
           \item for
             some (any)
             $x\in 
             cv_n$, the (projectivized) Outer space for $F_n$, the orbit $Hx \subset cv_n$
             is $\sigma-$quasiconvex with respect to the Lipschitz metric.
           \end{enumerate}
       
       \begin{rmark}
       The above Definition, while not explicit in \cite{dt1}, is implicit in Section 1.2 of that paper.
       
       Also, a word about the metric on $cv_n$ is in order. The statements in \cite{dt1} are made with respect to the {\em unsymmetrized} metric on outer space. However, convex cocompact subgroups have orbits lying in the thick part; and hence the unsymmetrized and symmetrized metrics are quasi-isometric to each other. We assume henceforth, therefore, that we are working with the symmetrized metric, to which the conclusions of \cite{dt1} apply via this quasi-isometry.
       \end{rmark}

           The following Theorem gives a characterization of  convex
           cocompact  subgroups in this context and is the analog of
           Theorem \ref{qi-coco}.

           \begin{theorem}\cite{dt1} \label{theo;dt14} 
             Let $H  $ be a  finitely generated subgroup of $Out(F_n)$ all whose 
             non-trivial elements  are atoroidal and fully irreducible. Then
             $H$  is convex cocompact 
             if and only if for
             some (any)
             $x\in 
             \FF_n$ (the free factor complex of $F_n$), the orbit $Hx \subset \FF_n$
             is qi-embedded in $\FF_n$.\label{qi-coco-out}\end{theorem}

           Dowdall and Taylor also show \cite[Theorem 4.1]{dt1} that any two
           quasi-geodesics in $cv_n$
           converging to the same point $p$ on the limit set $\Lambda_H$ (in
           $\partial cv_n$) of a convex cocompact subgroup $H$
           are asymptotic. More precisely, given $\lambda, \mu$ and $p \in
           \Lambda_H$ there exists $C_0(= C_0(\lambda, \mu, p))$ such that any
           two 
           $(\lambda, \mu)-$quasi-geodesics in $cv_n$ converging to
           $p$ are asymptotically $C_0-$close.
           As observed before in the context of Theorem \ref{alght-mcg0},
           this is adequate for the proof of Theorem \ref{alght-mcg0} to go through:

           \begin{theorem} \label{alght-out}  
             Let $G=Out(F_n)$, and $cv_n$ the Outer space for $G$ with
             a base point $z_0$.
             Then for $\sigma \geq 0$, there exists $C\geq 0$
             such that the following holds. \\ Let
             $H_1, \dots , H_n$
             be a finite collection of $\sigma-$convex cocompact subgroups of $G$.
             Suppose that there exist distinct cosets $\{
             g_mH_{\alpha_m} \}$ with $\alpha_m \in \{1, \dots ,
             n\}$, $m = 1, \dots, n$,
             such that
             $\cap_m g_mH_{\alpha_m}g_m^{-1}$ is infinite. Then there exists a point
             $z \in cv_n$
             such that the ball of radius
             $C$
             in $cv_n$
             intersects every image of $z_0$ by a coset
             $g_mH_{\alpha_m} z_0$. 

             Further, for any $i \in \{1, \dots , n\}$, there are
             only finitely many double cosets of the form
             $H_i g_i H_{\alpha_i}$
             such that
             $H_i \cap \bigcap_i  g_i H_{\alpha_i}g_i^{-1}$
             is infinite.

             The collection $\{ H_1, \dots , H_n\}$ has finite algebraic height.
           \end{theorem}

           Since an analog of the stability result of \cite{DuTa}
           in the context of Out$(F_n)$ is missing at the moment, we
           cannot quite get an analog
           of Theorem \ref{alght-mcg}.

           \subsection{Algebraic and geometric qi-intersection
             property: Examples} 
           In the Proposition below
           we shall put parentheses 
           around (geometric) to indicate that the statement holds for
           both the qi-intersection property
           as well as the 
           geometric qi-intersection property.
           \begin{prop}\label{prop;satisfiesqiip}
             \begin{enumerate}
             \item  Let $H$ be a quasiconvex subgroup of a hyperbolic
               group $G$, endowed
               with a locally finite word metric. Then,
               $\{H\}$ satisfies the uniform (geometric) qi-intersection property.

             \item Let $H$ be a relatively quasiconvex subgroup of a relatively
               hyperbolic group $(G, \calP)$.  Let $\calP_0$ be a set of conjugacy
               representatives of groups in $\calP$, and $d$ a word metric on $G$
               over a generating set $S= S_0\cup \calP_0$, where $S_0$ is
               finite. Then $\{H\}$ satisfies the uniform (geometric) qi-intersection property
               with respect to $d$.

             \item Let $H$ be a convex-cocompact subgroup of the Mapping Class
               Group $Mod(\Sigma)$ of an oriented closed surface $\Sigma$ of genus
               $\geq 2$. If $d$ is a word metric on $Mod(\Sigma)$ that makes it
               quasi-isometric to the curve complex of $\Sigma$, then
               $H$ satisfies the uniform (geometric) qi-intersection
               property
               with respect to $d$.

             \item Let $H$ be a convex-cocompact  subgroup of  $Out(F_n)$ for 
               some $n\geq 2$.  If $d$ is a word metric on $Out(F_n)$  that makes it
               quasi-isometric to the free factor complex of $F_n$,
               then $H$ satisfies the uniform (geometric)
               qi-intersection property
               with respect to $d$.
             \end{enumerate}
           \end{prop}

           \begin{proof}
             All four cases have similar proofs. Consider the first point.\\

 \noindent {\bf Case 1: $G$ hyperbolic, $H$ quasiconvex.}\\

                Let $h$ be the height of $H$ (which is finite by Theorem
             \ref{alght}): every $h+1$-fold intersection of conjugates of $H$ is finite, but some $h$-fold intersection is infinite.

             The first  conditions of Definition \ref{qiintn} and
             Definition \ref{geoqiintn}   follow from this finiteness and 
             Proposition \ref{prop;hyp_qc_have_fgh}
             and  Theorem \ref{short-intn}.

             The second condition of Definition \ref{geoqiintn}
             follows from Proposition \ref{prop;projections}.

We prove the second condition of Definition  \ref{qiintn}  (on mutual coboundedness of elements of $\CC\HH_i$) iteratively.

By Theorem \ref{short-intn}, there exists $C_h$ such that  two elements of $\CC\HH_h$ are $C_h$-quasiconvex in $(G,d)$. Let $D>0$. If $A$ and $B$ are two distinct such elements such that the projection of $A$ on $B$ has diameter greater than $D$, then there are $D/C_h$ pairs of elements $(a_i, b_i)$ in $A\times B$,  such that $a_i^{-1}b_i $  are elements of length at most $ 20\delta C_h$. Choose $N_0$ larger than the cardinality of finite subgroups of $G$. By a standard pigeon hole argument, if $D$ is large enough, there are $N_0$ such pairs for which  $a_i^{-1}b_i $ take the same value.   It follows that there are two   essentially distinct conjugates of elements of $\HH_h$ that intersect on a subset of at least $N_0$ elements, hence on an infinite subgroup. This contradicts the definition of height, and it follows that $D$ is bounded, and elements of $\CC\HH_h$ are mutually cobounded.  

We continue by descending induction. Assume that the second property of Definition \ref{qiintn}  is established for  
$\CC\HH_{i+1}$. By Proposition \ref{prop;criterion}  it follows that $(G,d_{i+1})$ is hyperbolic.
Let $\delta_{i+1}$ be its hyperbolicity constant.
By Proposition \ref{cobpersists},  there exists $C_i$ such that  two elements of $\CC\HH_{i}$ are $C_i$- quasiconvex in $(G,d_{i+1})$.

Again take  $A$ and $B$  two distinct  elements of  $\CC\HH_{i}$ such
that the   projection of $A$ on $B$ has diameter greater than $D>1000\delta_{i+1}$ for
$d_{i+1}$. Then there are at least $D/C_i$ pairs of elements $(a_i,b_i)$ in $A\times
B$,  such that $a_i^{-1}b_i $ is an element of length at most $
20\delta_{i+1} C_h$ for the metric $d_{i+1}$, 
and for all $i$ there exists $i'$ such that the segments
$[a_i,b_i], [a_{i'}, b_{i'}]$ are $(200\delta_{i+1})$-far from one
another. Apply the Proposition \ref{prop;unfolding_qc}  to each
geodesic $[a_i, b_i]$ to find quasi-geodesics $q_i$ from $a_i$ to
$b_i$ in $(G,d)$ (this can also be done by Lemma \ref{ea-strong}). We
know that in $(G,d)$,  $A$ and $B$ are quasiconvex (for the constant
$C_h$).  By their definition, and by hyperbolicity, the paths $q_i$
end at bounded distance of a shortest-point projection of $a_i$ to $B$
(for $d$). 
Therefore, since $(G,d)$ is
hyperbolic, and since the $q_i$ are far from one another for $d$, it
follows that the $q_i$ are actually short for the metric $d$ (shorter than $(200\delta
C_h)$).   Since there are  $D/C_i$ pairs of elements $(a_i,b_i)$,   by
the pigeon hole argument, there is an element $g_0$ (of length at most
$(200\delta
 C_h)$ in the metric $d$) such that for  $D/(C_i\times
 B_{G,d}(200\delta C_h) )$ such pairs, the difference $a^{-1} b$  equals $g_0$.  If $D$
is large enough, $D/(C_i \times B_{G,d}(200\delta
C_h) )$ is larger than the cardinality of the finite order elements of
$G$. It follows that the two essentially  distinct conjugates of elements
of $\HH_i$, corresponding to the cosets $A$ and $B$, intersect on  a
set of size larger than any finite subgroup of $G$ (and of diameter
larger than $3$ in $d_{i+1}$). Thus the
intersection is an infinite subgroup of $G$. This subgroup is necessarily among the conjugates of some $\HH_j$ for $j\geq i+1$, but therefore must have diameter $2$ in the metric $d_{i+1}$.\\

             \noindent {\bf Case 2: $G$ relatively hyperbolic, $H$
               relatively quasiconvex.}\\

               The geometric height of $H$ for the relative metric is
               finite, by Proposition \ref{prop;rh_or_gh}.  Let $h$ be
               its value. 
             The first points of Definition \ref{qiintn} follows from  this finiteness and  
             Theorem \ref{hruska-intn}.

                   The second point has a similar proof as the first case, except that the pigeon hole argument needs to be made precise because the relative metric $(G,d)$ is not locally finite.

Let $D>0$. If $A$ and $B$ are two
 distinct  elements  of $\CC\HH_h$  such that the projection of $A$ on $B$ has diameter
 greater than $D$,  then there are $D/C_H$ pairs of elements $(a_i,
 b_i)$ in $A\times B$,  such that $a_i^{-1}b_i $  are elements
 of length at most $ 20\delta C_h$.  Moreover, if $D>100\delta C_h$, 
for each  $[a_i,b_i]$, there is $[a_j, b_j]$ such that both segments
are short (for $d$) and  are
 at distance at least $(50\delta C_h)$ from each other. It
 follows that, in the Cone-off Cayley graph of $G$,  the maximal angle
 of $[a_i, b_i]$ at the cone vertices  is uniformly  bounded
 by $(100\delta C_h) + 2(2C_h+5\delta)$. Indeed, consider
 $\alpha$ and $\beta$ 
 quasi-geodesic paths in $A$ and $B$ respectively, from 
 $a_i$ to $a_j$ and from $b_i$ to $b_j$.  By hyperbolicity and
 quasi-geodesy, at distance $30\delta C_h$ from $a_i$ and
 $b_i$, there is a path of length $2(2C_h+5\delta)$ joining $\alpha$ to
 $\beta$. Being too short, this path cannot possibly intersect $[a_i,
 b_i]$. There is thus a path from $a_i$ to $b_i$ of length at most
 $2\times (30\delta C_h) +  2(2C_h+5\delta)$ that does not intersect
 $[a_i, b_i]$ outside its end points. It follows indeed that the
 maximal angle of  $[a_i, b_i]$  is at most  $2\times (30\delta C_h) +
 2(2C_h+5\delta)+ 20\delta C_h$.

From this bound on angles, we may use the fact that the angular metric
at each cone point is locally finite (by definition of relative
hyperbolicity) and the bound on the length in the
metric $d$,  to get that  all the elements $a_i^{-1}b_i $ are in a
finite set, independent of $D$.  We can now use the  pigeon hole argument, as in the hyperbolic case, and conclude similarly that $D$ is bounded.

The rest of the argument is also  by descending induction. Assume that
the second property of Definition \ref{qiintn}   is established for
$\CC\HH_{i+1}$.  
We proceed in a very similar way as in the hyperbolic case, with  the
difference is that, after establishing that the paths $q_i$ are small
for the metric $d$, one needs to check that their angles at cone
points are bounded, which is done by the argument we just used. We
provide the details now.

By Proposition \ref{prop;criterion}  it follows that $(G,d_{i+1})$ is hyperbolic.
Let $\delta_{i+1}$ be its hyperbolicity constant.
By Proposition \ref{cobpersists},  there exists $C_i$ such that  two elements of $\CC\HH_{i}$ are $C_i$- quasiconvex in $(G,d_{i+1})$.

Take  $A$ and $B$  two distinct  elements of  $\CC\HH_{i}$ such
that the   projection of $A$ on $B$ has diameter greater than
some constant $D$ for
$d_{i+1}$. Take a quasigeodesic in the projection of $A$ on $B$, of
length $D$.      Then there are at least $D/C_i$ pairs of elements $(a_i,b_i)$ in $A\times
B$, with $b_i$ on that quasigeodesic, and  such that $a_i^{-1}b_i $ is an element of length at most $
20\delta_{i+1} C_h$ for the metric $d_{i+1}$, 
and for all $i$ there exists $i'$ such that the segments
$[a_i,b_i], [a_{i'}, b_{i'}]$ are $(200\delta_{i+1})$-far from one
another. Apply the Proposition \ref{prop;unfolding_qc} (or
alternatively \ref{ea-strong})  to each
geodesic $[a_i, b_i]$ to find quasi-geodesics $q_i$ from $a_i$ to
$b_i$ in $(G,d)$. We
know that in $(G,d)$,  $A$ and $B$ are quasiconvex (for the constant
$C_h$).  By their definition, and by hyperbolicity, the paths $q_i$
end at bounded distance of a shortest-point projection of $a_i$ to $B$
(for $d$). 
Therefore, since $(G,d)$ is
hyperbolic, and since the $q_i$ are far from one another for $d$, it
follows that the $q_i$ are actually short for the metric $d$ (shorter than $(200\delta
C_h)$).   
By the argument used at the initial step of the descending induction, we also have an uniform upper bound on the
maximal angle of these paths, and therefore on the number of elements
of $G$ that label one of the paths $q_i$.

Since there are  $D/C_i$ pairs of elements $(a_i,b_i)$, if $D$ is
large enough,  by
the pigeon hole argument, there is an element $g_0$ (of length at most
$(200\delta
C_h)$ in the metric $d$), and a pair $(a_{i_0}, b_{i_0})$, such that
$a^{-1}_i b_i = g_0$ and such that 
for  $1000\delta_{i+1} C_i$  
other such pairs $(a_j, b_j)$, the difference $a^{-1}_j b_j$ is also
equal to $g_0$. The intersection of  two essentially  distinct conjugates of elements
of $\HH_i$, corresponding to the cosets $A$ and $B$, thus contains
$a_{i_0}^{-1} a_j$ for all those indices $j$.  There are indices $j$  for which $a_{i_0}^{-1} a_j$ labels
a quasi-geodesic paths in $A$ of length at least 
 $1000\delta_{i+1} C_i  $.   Such an element is either loxodromic,
 or elliptic with fixed point at the midpoint $[a_{i_0}, a_j]$.  But
 if all of them are elliptic, 
 for two indices $j_1, j_2$, we get two different fixed points, hence the
 product of the elements $a_{i_0}^{-1} a_{j_1} a_{i_0}^{-1} a_{j_2}$   is loxodromic.

  This element is in the intersection of conjugates
of elements of  $\HH_i$,   thus is in a subgroup  among the conjugates
of some $\HH_j$ for $j\geq i+1$, but therefore must have diameter $2$
in the metric $d_{i+1}$, and cannot contain loxodromic elements. This
is thus a contradiction.   \\

             \noindent {\bf Cases 3 and 4: $G\, = \, Mod(\Sigma)$ or
               $Out(F_n)$, $H$ convex cocompact.}\\

            Consider the Teichmuller metric on Teichmuller space
            $(Teich(\Sigma),d_T)$ 
            and the 
             (symmetrization of the) Lipschitz metric on Outer space
             $(cv_n,d_S)$ 
             respectively for
             $Mod(\Sigma)$
             and $Out(F_n)$. 
             Though $Teich(\Sigma)$ and $cv_n$ are
             non-hyperbolic, they 
             are proper metric spaces.

             For the mapping class group $Mod(\Sigma)$, the curve
             complex $(CC(\Sigma), d)$ is hyperbolic and  quasi-isometric to $(Mod(\Sigma),d)$, where $d$ is
             the word-metric on $Mod(\Sigma)$ obtained by taking as
             generating set a finite generating
             set of $Mod(\Sigma)$ along with {\it all} elements of
             certain sub-mapping class groups (see \cite{masur-minsky}).

             Similarly for $Out(F_n)$, the free factor complex
             $(\FF_n, d) $
             is hyperbolic, and is quasi-isometric to $(Out(F_n),d)$
             for a certain word metric over an infinite generating
             set  (\cite{BF-ff}).   
             This establishes that the hypotheses in the statements of
             Cases 3 and 4 are not vacuous.

             Recall that if a subgroup  $H$ of $Mod(\Sigma)$ or $Out(F_n)$
             is $C$-convex co-compact, then by Theorem \ref{theo;klh} (and
             \ref{theo;dt14})   the orbit of a base point
             in $CC(\Sigma)$ (or $\FF_n$) is a quasi-isometric image
             of the orbit of a base point in Teichmuller space.

             Finiteness of height of convex cocompact subgroups
             follows from Theorems \ref{alght-mcg} and \ref{alght-out}
             for 
             $G\, = \, Mod(\Sigma)$  and $Out(F_n)$ respectively.
             The first  condition of Definition \ref{qiintn} now follows from Theorems
             \ref{short-intn-coco}.

             We now proceed with proving the second  condition of Definition \ref{qiintn}.
             
             We first remark that, given $C$, there exists $\Delta, C'$ such that  if $A,B$ are cosets of $C$-convex
             co-compact subgroups, and if $a_1, a_2 \in A$, $b_1,
             b_2\in B$ are such that, in $CC(\Sigma)$, $d(a_1, b_1)$
             and $d(a_2, b_2) $ are at most $10C\delta$ and that  $d(a_1, a_2)$
             and $d(b_1,b_2)$ are larger than $\Delta$ then, $d_T(
             a_i, b_i ) \leq C' $ for both $i=1,2$. Indeed, by
             definition of convex cocompactness, the segment $[a_1,
             a_2 ]$ in Teichmuller space maps on a parametrized
             quasi-geodesic in the curve complex. A result of
             Dowdall Duchin and Masur ensures that
             Teichmuller geodesics that make progress in the curve
             complex, are contracting in Teichmuller space    
             \cite[Theorem A]{DDM} (see  the formulation done and proved in 
             \cite[Prop. 3.6]{DH}). Thus the
             segment   $[a_1,
             a_2 ]$   is contracting in Teichmuller space: any Teichmuller
             geodesic whose projection in the curve complex
             fellow-travels that of   $[a_1,
             a_2 ]$ has to be uniformly close to $[a_1,
             a_2 ]$. Applying that to 
             the segment $[b_1, b_2]$, it follows
             that it must remain at bounded distance (for Teichmuller
             distance) from $[a_1, a_2]$, as demanded. 

             A
             similar statement is valid for $Out(F_n)$  with the
             objects that we introduced, it suffice to use \cite[Prop.
             4.17]{DH}, an arrangement of Dowdall-Taylor's result
             \cite{dt1}, in place of the Dowdall-Duchin-Masur criterion.

             With this estimate, one can easily  adapt  the proof of the
             first case to get the result.

           \end{proof}

           \section{From quasiconvexity to graded relative hyperbolicity}

           Recall that we  defined  graded
           geometric relative hyperbolicity in
           Definition \ref{ggrh}.

           \subsection{Ensuring  geometric graded relative hyperbolicity}

           \begin{prop} \label{ghqiimpliesggrh-a}
             Let $G$ be a group, $d$ a word metric on $G$
             with respect to some (not necessarily finite)
             generating set, such that $(G,d)$ is
             hyperbolic.
             Let $H$ be a subgroup of $G$. 
             If $\{H\}$ has finite 
             geometric height for $d$ and  has the  uniform qi-intersection
             property, then $(G,\{H\},d)$ has   graded
             geometric relative hyperbolicity.
           \end{prop}

           \begin{proof} 
             As in Definition \ref{qiintn}, $\HH_n$ denotes the collection of
             intersections of $n$
             essentially distinct conjugates of $H$. Let
             $(\HH_n)_0$ denote  a set of conjugacy
             representatives of $(\HH_n)$ that are $C_1$-quasiconvex, and let
             $\CC\HH_n$ denote the collection of cosets of elements of $(\HH_n)_0$.
             Let $d_n$ be the metric on $X=(G,d)$ after electrifying the elements of
             $\CC\HH_n$. 

             By Definition \ref{qiintn}
             and Remark \ref{rem:cobdd}, for all $n$, all elements of $\CC\HH_n$ and
             of $\CC\HH_{n+1}$  are
             $C_1$-quasiconvex in $(G,d)$. Therefore, by Proposition
             \ref{cobpersists}, all elements of  $\CC\HH_n$ are
             $C'_1$-quasiconvex  in $(G,d_{n+1})$ for
             some $C'_1$ depending on the hyperbolicity of $d$, and on $C_1$.

             By   Definition \ref{qiintn},  $\CC\HH_n$ is mutually cobounded in
             the metric $d_{n+1}$.     Proposition
             \ref{prop;criterion} now shows that the horoballification of $(G,d_{n+1})$
             over $\CC\HH_{n}$ is hyperbolic, for all $n$.  Proposition
             \ref{prop;from_horo_to_he} then guarantees that  $\CC\HH_{n}$ is
             coarsely hyperbolically embedded in $(G,d_{n+1})$. Since $H$ is assumed to
             have finite  
             geometric height,     $(G,\{H\},d)$ has 
             graded geometric   relative hyperbolicity.
           \end{proof}

           \subsection{Graded relative hyperbolicity for quasiconvex subgroups}

           \begin{prop}\label{qcimpliesgrh}
             Let $H$ be a  quasiconvex subgroup of a hyperbolic group $G$,  with a
             word metric $d$ (with respect to a finite generating set). Then the
             pair $(G,\{H\})$ has   
             graded geometric
             relative hyperbolicity, and graded relative hyperbolicity.
           \end{prop}

           \begin{proof} 
             For the word metric $d$ with respect to a finite generating set, 
  graded geometric relative hyperbolicity agrees with the notion of graded 
             relative hyperbolicity (Definition \ref{grh}).

             By Theorem \ref{alght}, $H$ has finite height. By Proposition
             \ref{prop;satisfiesqiip} it satisfies the uniform qi-intersection
             property \ref{qiintn}.  Therefore, by   Proposition
             \ref{ghqiimpliesggrh-a},  the pair $(G,\{H\})$ has 
             graded relative
             hyperbolicity.

             Finally, note that since the word metric we use is locally finite, and all $i-$fold intersections are quasiconvex,  graded relative
             hyperbolicity follows.
           \end{proof}

           \begin{prop} \label{relqcimpliesgrh}
             Let $(G,\PP)$ be a finitely generated relatively
             hyperbolic group. Let $H$ be a relatively quasiconvex
             subgroup. Let $S$ be a finite relative generating set of
             $G$ (relative to $\PP$) and let
             $d$ be the word metric with respect to $S \cup \PP$.  Then  $(G,\{H\},d)$ has 
             graded relative hyperbolicity as well as  
        graded geometric relative hyperbolicity. 
           \end{prop}

           \begin{proof}
             The proof is similar to that of Proposition \ref{qcimpliesgrh}. By
             Theorem \ref{relht}, $H$ has finite relative height,  hence it has
             finite  
      geometric
             height for the relative metric (see Example
             \ref{prop;rh_or_gh}).

             Next, by  Proposition \ref{prop;satisfiesqiip}, $H$ satisfies the uniform
             qi-intersection property for a relative metric, and 
             graded geometric relative hyperbolicity follows from
             Proposition \ref{ghqiimpliesggrh-a}.

             Again, since $G$ has a word metric with respect to a
             finite relative generating set, and $H$ and all $i-$fold
             intersections are relatively
             quasiconvex as well, the above argument furnishes graded
             relative hyperbolicity as well. 
           \end{proof}

           Similarly, replacing the use of Theorem
           \ref{alght} by Theorems \ref{alght-mcg0} and \ref{alght-out}, 
           one  obtains the following.

           \begin{prop}\label{cocoimpliesgrh0} Let $G$ be the mapping
             class group $Mod(S)$ (respectively $Out(F_n)$). Let $d$
             be a word metric
             on $G$ making it quasi-isometric to the curve complex
             $CC(S)$ (respectively the free factor complex $\FF_n$).
             Let $H$ be a  convex cocompact subgroup of  $G$. Then  $(G,\{H\},d)$ has
             graded   relative hyperbolicity. 
           \end{prop}

           Again, replacing the use of Theorem
           \ref{alght} by Theorem \ref{alght-mcg}, we obtain:

           \begin{prop}\label{cocoimpliesgrh} Let $G$ be the mapping
             class group $Mod(S)$. Let $d$ be a word metric
             on $G$ making it quasi-isometric to the curve complex
             $CC(S)$.
             Let $H$ be a  convex cocompact subgroup of  $G$. Then  $(G,\{H\},d)$ has 
             graded  geometric relative hyperbolicity. 
           \end{prop}

           \begin{remark} Since we do not have an exact (geometric)
             analog of Theorem \ref{alght-mcg} for Out($F_n$) (more
             precisely an analog of the stability result of
             \cite{DuTa}) as of now, we have to content ourselves with
             the slightly weaker Proposition \ref{cocoimpliesgrh0} for  Out($F_n$).
           \end{remark}

           \section{From graded relative hyperbolicity to quasiconvexity}

           \subsection{A Sufficient Condition}

           \begin{prop}\label{ggrhtoqc}
             Let $G$ be a group and $d$ a hyperbolic word metric  with respect to a (not
             necessarily finite) generating set.  
             Let $H$ be a subgroup such that $(G,\{H\},d)$ has 
              graded geometric
             relative hyperbolicity.
             Then  $H$ is quasiconvex in $(G,d)$.
           \end{prop}

           \begin{proof}  Assume $(G,\{H\},d)$ has 
             graded geometric relative
             hyperbolicity as in Definition \ref{ggrh}. Then $H$ has finite geometric height in
             $(G,d)$. Let $k$ be this height. Thus, $\HH_{k+1}$ is a collection
             of uniformly bounded subsets,  and
             $d_{k+1} $ is quasi-isometric to $d$.  It follows   
             that $(G,d_{k+1})$ is hyperbolic.

             Further, by  Definition \ref{ggrh}, 
             $\HH_k$  is hyperbolically embedded   in
             $(G,d_{k+1})$. This means in particular that the
             electrification $(G,d_{k+1})^{el}_{\HH_k} $
             is  hyperbolic. Since $(G,d_k)$ is quasi-isometric to
             $(G,d_{k+1})^{el}_{\HH_k} $ (being
             the restriction of the metric on $G$)   
             it follows that $(G,d_k)$ is hyperbolic as well. Further,
             by Corollary \ref{cor-retraction}  the elements
             of $\HH_k$, are uniformly quasiconvex in $(G,d_{k+1})$.

             We now argue by descending induction on $i$.

             \noindent {\it The inductive hypothesis for $(i+1)$:}\\
             We assume that  $d_{i+1}$ is a hyperbolic metric on $G$,
             and that  there is a constant $c_{i+1}$ such that, for
             all $j\geq 1$ 
             the  elements of $  \HH_{i+j}$
             are uniformly $c_{i+1}$-quasiconvex in $(G,d)$.

             We assume the inductive hypothesis for $i+1$ ({\it i.e.} as
             stated), and we now prove it for $i$.

             Of course, we also assume, as in the statement of the Proposition,
             that
             $\HH_{i}$   is coarsely hyperbolically embedded in
             $(G,d_{i+1})$.  Hence $d_i$ is a hyperbolic
             metric on $G$.

             We will now  check that the assumptions of Proposition
             \ref{prop;unfolding_qc}
             are satisfied for $(X,d) = (G, d_{i+1})$, $\YY =
             \HH_{i+1}$, and $H_{i,\ell}$ 
             arbitrary in $\HH_i$.

             Elements of $\HH_{i}$
             in $(G,d_{i+1})$ are uniformly quasiconvex in $(G,d_{i+1})$: this
             follows  from Corollary \ref{cor-retraction}. We will write
             $C_i$ for their quasiconvexity constant. 

             A second step is to check that, for some uniform $\Delta_0$ and
             $\epsilon$, for all $\Delta>\Delta_0$,    when an element
             $H_{i,\ell}$ of $\HH_{i}$
             $(\Delta,\epsilon)$-meets an item of $\HH_{i+1}$, then
             $H^{+\epsilon \Delta}$  contains
             a quasigeodesic  between the meeting points in $H$.   Thus, fix
             $\epsilon <1/100$, and take $\Delta_0 $ 
             larger than $20$ times the thickening constants for the definition
             of elements in $\HH_i$ (which is possible by finiteness
             of number of orbits of i-fold intersections).  
             Assume
             $H_{i,\ell}$   $(\Delta, \epsilon)$-meets $Y \in \HH_{i+1}$.
             Then, by definition of $i-$fold intersections
             \ref{def:ifoldinter}, and Proposition \ref{geointnsstable},
             either the pair of meeting points is in  an item of
             $\HH_{i+1}$ inside $H_{i,\ell} $, or
             $Y\subset H_{i,\ell}$. 
             In both cases, by the inductive assumption,
             there is a path in $H_{i,\ell}^{+\epsilon \Delta}$ between the meeting
             points  in   $H_{i,\ell}$ that is a quasigeodesic for
             $d$.  Hence the second assumption
             of Proposition \ref{prop;unfolding_qc} is satisfied. 

             We can thus conclude by that proposition that $H_{i,\ell}$ is quasiconvex in $(G,d)$ for a
             uniform constant, and therefore the inductive assumption holds for $i$.

             By induction it is then true for $i=0$, hence the first
             statement of the Proposition holds, i.e.
             quasiconvexity follows from 
              graded geometric relative hyperbolicity.

           \end{proof}

           We shall deduce various consequences of Proposition
           \ref{ggrhtoqc} below.  However, before we proceed, we need
           the following observation
           since we are dealing with spaces/graphs that are not necessarily proper.

           \begin{obs} \label{qctoqi} Let $X$ be a (not necessarily
             proper) hyperbolic graph. For all $C_0\geq 0$, there
             exists
             $C_1 \geq 0$ such that the following holds: \\ Let $H$ be
             a hyperbolic group acting uniformly properly
             on $X$, i.e. for all $D_0$ there exists $N$ such that for
             any $x \in X$, any $D_0$ ball in $X$ contains at most $N$
             orbit points of $Hx$.
             Then a $C_0-$quasiconvex orbit of $H$ is
             $(C_1,C_1)-$quasi-isometrically embedded in $X$.
           \end{obs}

           Combining  Proposition \ref{ggrhtoqc} with Observation
           \ref{qctoqi} we obtain the following:

           \begin{prop}\label{ggrhtoqi}
             Let $G$ be a group and $d$ a hyperbolic 
             word metric  with respect to a (not necessarily finite) generating set.
             Let $H$ be a subgroup such that
             \begin{enumerate}
             \item  $(G,\{H\},d)$ has graded geometric relative hyperbolicity.
             \item The action of $H$ on $(G,d)$ is uniformly proper.
             \end{enumerate}
             Then $H$ is hyperbolic and 
             $H$  is qi-embedded in $(G,d)$.
           \end{prop}

           \begin{proof} 
             Quasi-convexity of $H$ in $(G,d)$ was established in Proposition \ref{ggrhtoqc}.
             Qi-embeddedness 
             of $H$ follows from Observation \ref{qctoqi}. Hyperbolicity of $H$ is 
             an immediate consequence. \end{proof}

           \subsection{The Main Theorem} We assemble the pieces now to
           prove the following main theorem of the paper.

           \begin{theorem}\label{hypcharzn} Let $(G,d)$ be one of the following:
             \begin{enumerate} 
             \item $G$ a hyperbolic group and $d$ the word metric with
               respect to a finite generating set $S$.
             \item $G$ is finitely generated and hyperbolic relative
               to $\PP$, $S$ a finite relative generating set,  and
               $d$ 
               the word metric with respect to $S \cup \PP$.
             \item $G$ is the mapping class group $Mod(S)$ and $d$ the
               metric obtained by electrifying the  subgraphs
               corresponding to
               sub mapping class groups so that $(G,d)$ is
               quasi-isometric to the curve complex $CC(S)$.
             \item $G$ is $Out(F_n)$ and $d$ the metric obtained by
               electrifying the  subgroups corresponding to
               subgroups that stabilize proper free factors so that
               $(G,d)$ is quasi-isometric to the  free factor complex
               $\FF_n$.
             \end{enumerate}
             Then (respectively)
             \begin{enumerate} 
             \item $H$ is quasiconvex if and only if $(G,\{H\})$ has
               graded geometric  relative hyperbolicity.
             \item $H$ is relatively quasiconvex if and only if
               $(G,\{H\},d)$ has graded geometric relative
               hyperbolicity.
             \item $H$ is convex cocompact in $Mod(S)$  if and only if
               $(G,\{H\},d)$ has graded geometric relative
               hyperbolicity and the  action of
               $H$ on the curve complex is uniformly proper.
             \item $H$ is convex cocompact in  $Out(F_n)$  if and only
               if $(G,\{H\},d)$ has graded geometric relative hyperbolicity
               and the  action of
               $H$ on the free factor complex is uniformly proper.
             \end{enumerate}
           \end{theorem}
           
           \begin{proof} The forward implications of quasiconvexity to graded 
             geometric relative hyperbolicity in the first 3 cases are proved by
             Propositions \ref{qcimpliesgrh}, \ref{relqcimpliesgrh},
             \ref{cocoimpliesgrh0} and \ref{cocoimpliesgrh}
             and case 4 by Proposition \ref{cocoimpliesgrh0}. In cases
             (3) and (4) properness of the action of $H$ on 
             the curve complex follows from convex cocompactness.
             
             We now proceed with the reverse implications.  Again, the
             reverse implications of (1) and (2) are direct
             consequences
             of Proposition \ref{ggrhtoqc}. 
             
             The proofs of the reverse implications of (3) and (4) are
             similar.  Proposition \ref{ggrhtoqi} proves that any
             orbit of $H$ on
             either the curve complex $CC(S)$ or the free factor
             complex $\FF_n$ is qi-embedded. Convex cocompactness now
             follows from
             Theorems \ref{qi-coco} and \ref{qi-coco-out}. 
           \end{proof}
           
           \subsection{Examples} We give a couple of examples below to show that 
           finiteness of geometric height does not necessarily follow from quasiconvexity.
           
           \begin{example} {\rm Let $G_1 = \pi_1(S)$ and $H = <h>$ be
               a cyclic subgroup corresponding to a simple closed
               curve. Let $G_2 = H_1 \oplus H_2$
               where each $H_i$ is isomorphic to $\mathbb{Z}$. Let $G
               =G_1 \ast_{H=H_1} G_2$. Let $d$ be the metric obtained
               on $G$ with respect to some finite generating
               set along with all elements of $H_2$. Then $G_1$ is
               quasiconvex in $(G,d)$, but $G_1$  does not have finite
               geometric height.} 
           \end{example}
           
           Note however, that the action of $G_1$ on $(G,d)$ is not
           acylindrical. We now furnish another example to show that 
           graded geometric relative hyperbolicity does not
           necessarily follow from quasiconvexity even if we assume
           acylindricity.

           \begin{example}{\rm Let $G = \langle a_i, b_i: i \in
               \natls, a_{2i}^{b_i} = a_{2i-1} \rangle$ and let $F$ be
               the (free) subgroup
               generated by $\{ a_i\}$. Then $F^{b_i}\cap F =
               <a_{2i-1}>$ for all $i$. Let $d$ be the word metric on
               $G$ with respect to the generators
               $a_i, b_i$. Then the action of $F$ on $(G,d)$ is
               acylindrical and $F$ is quasiconvex. However there are
               infinitely many double coset representatives
               corresponding to $b_i$ such that $F^{b_i}\cap F $ is infinite.}
           \end{example}

\section*{Acknowledgments} This work was initiated during a visit of
the second author to  Institut Fourier in Grenoble during June 2015 and carried on while visiting Indian Statistical
Institute, Kolkata.
He thanks the Institutes for their hospitality.

The authors were supported in part by the National Science Foundation 
under Grant No. DMS-1440140 at  the Mathematical Sciences Research Institute in Berkeley
during Fall 2016, where this research was completed.

We thank the referee for a detailed and careful reading of the manuscript and for several extremely helpful and perceptive comments.
\bibliography{qcdm_26mm}
\bibliographystyle{alpha}

\newpage

\part*{Corrigendum to {\it Height graded relative hyperbolicity and quasiconvexity}}

\setcounter{section}{0}

\section{Saturated graded relative hyperbolicity}
There is an unfortunate mistake in the statement and the proof of Proposition 5.1 of \cite{DM17}. This affects one direction of the implications of the main theorem 6.4 (also visible as 1.4). This also affects 5.2, 5.3, 5.4, and 5.5 that are essentially the different specific cases of  the statement of 6.4. Unless explicitly mentioned the references to Theorem numbers, etc.\ refer to \cite{DM17}.

We explain the problem, and give a correction.  

\subsection{Main modifications and corrections}

Proposition  5.1 of \cite{DM17} must  be replaced by Proposition \ref{prop;new51} below.

  The conclusions of Propositions 5.2, 5.3, 5.4 and 5.5 should be changed to:  $(G,\{H\},d)$ have the {\bf saturated} (geometric) graded relative hyperbolicity (as defined below in Section \ref{sec;satu}).

Subsequently, in the conclusions of Theorems 1.4 and 6.4,  the graded relative hyperbolicity should be changed to  {\bf saturated}  graded relative hyperbolicity as stated in the section \ref{sec;conclusion} below.

\subsection{Setting of Proposition 5.1}

Recall the setting of Proposition 5.1. 

Let $G$ be a group, $d$ a word metric on $G$ with respect to some generating set (not necessarily finite), such that $(G,d)$ is hyperbolic. Let $H$ be a subgroup of $G$. Assume that $H$ has finite geometric height (Definition 4.1) and uniform q.i intersection property (Definition 3.9).

The original version of Proposition 5.1 then asserted that $(G, \{H\}, d)$ had the geometric graded relative hyperbolicity property. This is false in general. We propose here a correction of the statement and of the argument, which uses a natural operation, the saturation of a quasiconvex subgroup.

\subsection{Definition of saturation}\label{sec;satu}

Let $G$ be a relatively hyperbolic group, with $d$ the relative word metric,  and $A$ a relatively quasiconvex subgroup (i.e quasiconvex in $(G,d)$). Assume that $A$ is not parabolic. 
The saturation $A_s$ of $A$ in $(G,d)$  is the stabilizer of the limit set $\Lambda A$ of $A$ in $\partial (G,d)$.   If $A$  is  infinite (or even of larger cardinality than any  finite subgroup of $G$ that is not in a single parabolic group) yet bounded in the metric $(G,d)$ we simply say that $A_s$ is the (unique) maximal  parabolic subgroup containing $A$. Let us observe that $A_s$ is its own normalizer, contains $A$ and remains at bounded distance from $A$ in $(G,d)$. 

More generally, if $(G,d)$ is a group with a hyperbolic word metric, $A$ a subgroup, quasi-convex   for $(G,d)$, and of infinite diameter,  then $A_s$ is the stabilizer of the limit set  $\Lambda A$ of $A$ in $\partial (G,d)$. 

This later circumstance will be used for convex cocompact groups, which are quasi-isometrically embedded in $(G,d)$, hence either finite or of infinite diameter.

Let us define the saturated (geometric) graded relative hyperbolicity property,  as a variation of  Definition 4.3, %
as follows. Note that only the second condition is changed compared to the original definition 4.3.

  Let $G$ be a group, $d$ the word metric with respect to some (not necessarily finite) generating set and $\calH$ a finite collection of subgroups. Let $\calH_i$  be the collection of all i-fold conjugates of $\calH$. Let $(\calH_i)_0$ be a choice of conjugacy representatives  and $\calC \calH_i$ the set of left cosets of elements of $(\calH_i)_0$.

Let also $\calS\calH_i$  be the collection of saturations of the groups in $\calH_i$ for the metric $d$. 
Let  $(\calS\calH_i)_0$ be a choice of conjugacy representatives of these groups, and   $\calC\calS\calH_i$  the family of left cosets of   $(\calS\calH_i)_0$. Let $d^s_i$ be the metric on $G$ obtained from $d$ by electrifying the elements of   $\calC\calS\calH_i$.    Let  $\calC\calS\calH_\mathbb{N}$ be the graded family  $(\calC\calS\calH_i)_{i\in \mathbb{N}}$. We say that $(G,d)$ has  saturated (geometric) graded relative hyperbolicity  with respect to  $\calC\calS\calH_\mathbb{N}$ if 

(1)   $\calH$ has (geometric) height $n$ for some $n \in \bbN$, and for each $i$ there are finitely   many orbits of $i$-fold intersections,
 
(2)  for all $i$, $\calC\calS\calH_i$ coarsely hyperbolically embeds in $(G,d^s_i)$, 

(3)  there is $D_i$ such that all items of $\calC \calH_i$ are $D_i$-coarsely path connected in $(G, d)$.

 We say that $(G,d)$ has  saturated graded relative hyperbolicity  with respect to  $\calC\calS\calH_\mathbb{N}$ if (2) and (3) are true, and (1) is true for the algebraic height.

\subsection{Uniform qi intersection property passes to saturations}

Recall that uniform qi intersection property was defined in Definition 3.9.  Let us first notice the following, that essentially states that this property  passes to saturations.

\begin{lembis}\label{lem;saturateUqiIP} Assume that $(G,d)$ is hyperbolic.
  If $H$ has uniform qi intersection property in $(G,d)$ and if  $\calC\calS\calH_\mathbb{N}$ is obtained as in the definition above, then for all $i$ there exists $C^s_i$ such that  each element of  $\calC\calS\calH_i$  is quasi-convex in $(G,d)$, and such that if $A_0, B_0$ are in $  (\calS\calH_i)_0$ and if $\Pi_{B_0} (gA_0)$ has diameter larger than $C^s_i$ for the metric $d$, then $ gA_0g^{-1}\cap B_0$ has diameter  larger than $C^s_i $ for $d$.
\end{lembis} 

\begin{proof} We first prove the uniform quasi-convexity of elements of $  (\calS\calH_i)_0$.  The groups in $(\calH_i)_0$ are uniformly quasi-convex in $(G,d)$, by the first point of Definition 3.9. Each  group $A_s$ in $(\calS\calH_i)_0$ has a subgroup $A$ in   $(\calH_i)_0$  which is co-bounded in $(G,d)$. By hyperbolicity of  $(G,d)$ any geodesic between two points of $A_s$  is close to a geodesic between two points of $A$, which itself remains at bounded distance from $A$ hence from $A_s$. This proves the first point.

For the second point, assume the contrary: for all $C_i$ we can find $A_s, B_s$  in $  (\calS\calH_i)_0$   and $g$  (all depending on $C_i$) such that  $\Pi_{B_s} (gA_s)$ has diameter larger than $C_i$ for the metric $d$, but $ gA_sg^{-1}\cap B_s$ has diameter  smaller than $C_i $ for $d$. Consider  elements $A, B$ in $(\calH_i)_0$ of which $A_s$ and $B_s$ are the saturations. 
  Of course  $ gAg^{-1}\cap B$ has diameter  smaller than $C_i $ for $d$. 
  On the other hand, we can see that   $\Pi_{B} (gA)$ must go to infinity with $C_i$. Indeed, take pairs of points  $a_0, a_1$ in $gA_s$ and $b_0, b_1$ in  $B_s$  realising the shortest point projection of $a_0,a_1$ respectively, and such that  $d(b_0,b_1)$ is larger than  $C_i$. Then, we may find $a_0', a_1' $ in $gA$ close to  $a_0, a_1$ (say at distance at most $D$), and consider their shortest point projection on $B$, say $b'_0, b'_1$.  Approximate the octagon $(b_0, a_0, a'_0, b'_0, b'_1, a'_1, a_1, b_1)$ by a tree, by hyperbolicity. Because both $B$ and $B_s$ are quasi-convex (with uniform constant over   $  (\calS\calH_i)_0$), the Gromov products of the consecutive sides at the vertices $b_0, b_1, b'_0, b'_1$ are uniformly bounded. One can then deduce that the central subsegment of $[b_0, b_1]$ of length  at least $C_i$ minus a universal constant, remains close to $[b'_0, b'_1]$. Thus,  $\Pi_{B} (gA)$ is larger than this quantity.

By uniform qi-intersection property for $\calH_i$, we than have a contradiction.

\end{proof}


\subsection{Correction to Proposition 5.1}
Then we can show a correct version of Proposition 5.1.

\begin{propbis}\label{prop;new51}
  Let $G$ be either a relatively hyperbolic group 
 with a relative word metric $d$, or a Mapping class group  with a word metric $d$ equivariantly quasi-isometric to the curve complex, or $Out(F_n)$  with a word metric $d$ equivariantly quasi-isometric to the free factor complex. %

Let $H$ be a subgroup of $G$. If $\{H\}$ has finite geometric height for $d$ and has the uniform qi-intersection property, then $(G,\{H\},d)$ has the saturated  (geometric) graded relative hyperbolicity property with respect to  $\calC\calS\calH_\mathbb{N}$.
\end{propbis}

\begin{proof}
We first assume that $G$ is relatively hyperbolic.
By Lemma \ref{lem;saturateUqiIP},  
the elements    $\calC\calS\calH_i$ also are  $C^s_1$-quasi convex in $(G,d)$ for some uniform constant $C^s_1$.  

     By Proposition 2.11, all the elements of  $\calC\calS\calH_i$ are uniformly quasi-convex in $(G,d_{i+1})$. 

We now need to show that the elements of  $\calC\calS\calH_i$ are mutually cobounded for  $d^s_{i+1}$ (a property which fails in general for $\calC\calH_n$).

Assume, by contradiction, that the elements (cosets)  of $\calC\calS\calH_i$  are not mutually co-bounded for the metric $d^s_{i+1}$. For all $D$ there exists two essentially different cosets $A_s$ (which can be assumed in $(\calS\calH_i)_0$) and $gB_s$ (for $B_s \in (\calS\calH_i)_0$ as well), that have projection larger than $D$ on one another for the metric $d^s_{i+1}$. Recall that essentially different means that either $A_s\neq B_s$ or $g\notin B_s$.

 By Lemma \ref{lem;saturateUqiIP},  
for $D$ large enough, $A_s\cap gB_sg^{-1}$ has diameter larger than $D-2C_n$ for  $d^s_{i+1}$.

By definition of saturation, $A_s$ and $B_s$ are either bounded, or equal the stabilizers of their respective limit sets $\Lambda A_s, \Lambda B_s$ in the hyperbolic metric $d$.

In the case where both $A_s$ and $B_s$ are bounded, yet infinite (or even of sufficiently large cardinality), 
then we are in the relatively hyperbolic case, and these groups must be parabolic, and as saturations, they are equal to  maximal parabolic group containing $A$ and $B$.  Their intersection cannot be larger than a universal constant for $d$, contrary to our assumption.

 In the  case $A_s$ and $B_s$ are unbounded, we first observe that in that metric, $\Lambda A_s\cap  \Lambda gB_sg^{-1} = \Lambda (A_s\cap gB_sg^{-1})$.
This is  a result of  Yang \cite[Thm 1.1]{y12} for relatively hyperbolic groups.  %

For a convex cocompact subgroup of a Mapping Class Group, (respectively of $Out(F_n)$), this observation can be derived from the fellow traveling property of a thick part of Teichmuller space (Rafi \cite{rafi-gt}) (respectively Dowdall-Taylor \cite{dt1c,dt2c}), and that the weak hull of the group remains in a thick part. We sketch an argument for Mapping Class Groups. If $\xi$ is limit of $(a_n)$ (sequence in  $A_s$) and of $(gb_ng^{-1})$ (in $gB_sg^{-1}$), then after possible re-indexing of  subsequences,  and choosing a base point $x_0$ in Teichmuller space, $a_nx_0$ and $gb_ng^{-1}x_0$ remain at bounded distance. Thus infinitely often, $a_i^{-1}a_j = gb_i^{-1}b_jg^{-1}$, hence the intersection accumulates on $\xi$. %

Now, taking $A$ and $B$ in  $\calH_i$ of which $A_s, B_s$ are the saturations,    $\Lambda A_s = \Lambda A$, and  $\Lambda gB_sg^{-1} = \Lambda gBg^{-1}$.
Therefore, the saturation of  $ (A\cap gBg^{-1})$ contains $A_s\cap gB_sg^{-1}$, hence it has  diameter larger than $D-2C_n$ for  $d^s_{i+1}$.

It follows  that   $ (A\cap gBg^{-1})$  is not an essential $(i+1)$-fold intersection of conjugates of $H$. Writing $A$ and $B$ as  $i$-fold intersection themselves as $A= X_1 \cap \dots \cap X_i$, and  $gBg^{-1}= Y_1 \cap \dots \cap Y_i$, then     $A\cap gBg^{-1} =   X_1 \cap \dots  \cap X_i  \cap Y_1 \cap \dots \cap Y_i $, which can only contain $i$ essentially distinct conjugates.    We must conclude that after permutation of indices, $X_j=Y_j$ for all $j$, and that $gBg^{-1} = A$. Therefore, at the level of saturations,
 $gB_sg^{-1} = A_s$. Since $A_s, B_s$ are in  $(\calS\calH_i)_0$ which is a set of conjugacy representatives, we have that $A_s=B_s$. And therefore $g$ normalizes $B_s$.

This  is precisely where the mistake was: had we taken $B$ in $(\calH_i)_0$, we could not have concluded that $g\in B$. But now, $B_s\in (\calS\calH_i)_0$, which consists of saturated subgroups, which are equal to their own normalizers. Thus, indeed,      we can conclude that  $g\in B_s$.  This contradicts  our original choice that $A_s$ and $gB_s$ are essentially different.

The end of the proof is now the use of Proposition 2.10, and 2.23, as in the original version of \cite{DM17}. 
\end{proof}
 
\subsection{Conclusion}\label{sec;conclusion}

We can proceed and correct statements 5.2 to  5.5. In each of them only the conclusion is changed to:  $(G,\{H\},d)$ have the {\bf saturated} geometric graded relative hyperbolicity.

The syntactical modification of the proof is straightforward, using the corrected proposition above in place of Proposition 5.1. 

The main results Theorems 1.4  and 6.4  need to be corrected as follows.

  \begin{thmbis}\label{hypcharzn_correct} 
             \begin{enumerate} 
             \item Let $G$ be  a hyperbolic group and $d$ 
               the word metric with
               respect to a finite generating set $S$.

               If a subgroup $H$ is  quasiconvex then  $(G,\{H\})$ has
               saturated  graded geometric  relative hyperbolicity. 

               If  $(G,\{H\})$ has  graded geometric  relative hyperbolicity, 
               then $H$ is quasiconvex.

             \item  Let $G$ be a  finitely generated group, hyperbolic relative
               to $\mathcal{P}$, $S$ a finite relative generating set,  and
               $d$ 
               the word metric with respect to $S \cup \mathcal{P}$.

               If a subgroup $H$ is relatively quasiconvex then $(G,\{H\},d)$ has saturated graded  geometric relative hyperbolicity. 

If $(G,\{H\},d)$   has  graded  geometric relative hyperbolicity, then  $H$ is relatively quasiconvex. 
               
             \item Let $G$ be a mapping class group $Mod(S)$ and $d$ the
               metric obtained by electrifying the  subgraphs
               corresponding to
               sub mapping class groups so that $(G,d)$ is
               quasi-isometric to the curve complex $CC(S)$.

               If a subgroup $H$ is convex cocompact in $Mod(S)$  then
               $(G,\{H\},d)$ has saturated graded geometric relative
               hyperbolicity. 

If  $(G,\{H\},d)$ has graded geometric relative
               hyperbolicity and the  action of
               $H$ on the curve complex is uniformly proper, 
               then  $H$ is convex cocompact in $Mod(S)$.

             \item Let $G$ be $Out(F_n)$ and $d$ the metric obtained by
               electrifying the  subgroups corresponding to
               subgroups that stabilize proper free factors so that
               $(G,d)$ is quasi-isometric to the  free factor complex
               $\mathcal{F}_n$.

               If a subgroup $H$ is convex cocompact in  $Out(F_n)$, then  
               $(G,\{H\},d)$ has saturated 
               graded geometric relative hyperbolicity. 

If $(G,\{H\},d)$ has
               graded geometric relative hyperbolicity
               and the  action of
               $H$ on the free factor complex is uniformly proper, then $H$ is convex cocompact in  $Out(F_n)$.

             \end{enumerate}
      
           \end{thmbis}
           
           \begin{proof} The forward implications of quasiconvexity to graded 
             geometric relative hyperbolicity in the first 3 cases are proved by
             the corrections above of Propositions 5.2, 5.3, 5.4, and 5.5.
             and case 4 by the correction of Proposition 5.4. 
             
             Reverse implications are those of Theorem 6.4, their proof is unchanged. 
             
           \end{proof}

\section*{Acknowledgment} We would like to thank Denis Osin for bringing the error in \cite{DM17} to our notice.


\end{document}